\documentclass[12pt]{article}

\usepackage{a4wide}
\usepackage{amsmath}
\usepackage{amssymb}
\usepackage{amsthm}
\usepackage{latexsym}
\usepackage{graphicx}
\usepackage[english]{babel}
\usepackage{makeidx}
\usepackage{arcs}
\usepackage{pst-node}
\usepackage{tikz-cd}

\newtheorem{dfn} [subsection]{Definition}
\newtheorem{obs} [subsection]{Remark}

\newtheorem{prop}[subsection]{Proposition}

\newtheorem{teor}[subsection]{Theorem}

\newtheorem{lema}[subsection]{Lemma}
\newtheorem{cor} [subsection]{Corollary}

\newcommand{\AAA}{\mathcal A}
\newcommand{\FF}{\mathcal F}

\newcommand{\GG}{\mathcal G}
\newcommand{\HH}{\mathcal H}
\newcommand{\KK}{\mathcal K}

\newcommand{\PP}{\mathcal P}

\newcommand{\Coverage}{\mathbf{Coverage}}
\newcommand{\Covering}{\mathbf{Covering}}
\newcommand{\fCovering}{\mathbf{f-Covering}}
\newcommand{\gCovering}{\mathbf{g-Covering}}
\newcommand{\Partition}{\mathbf{Partition}}
\newcommand{\Seth}{\mathbf{Set}}
\newcommand{\Fu}{\mathbf F}
\newcommand{\Gu}{\mathbf G}

\def\Imm{\operatorname{Im}}

\def\uu{\operatorname{u}}
\def\Gg{\operatorname{g}}
\def\id{\operatorname{id}}
\def\Id{\operatorname{id}}
\def\Ob{\operatorname{Ob}}
\def\Hom{\operatorname{Hom}}

\numberwithin{equation}{section}

\title{Geometrical isomorphisms between categories of fuzzy coverings and fuzzy partitions}

\author{Mircea Cimpoea\c s$^1$ and Adrian Gabriel Neac\c su$^2$}
\date{}

\begin{document}

\maketitle
\footnotetext[1]{Faculty of Applied Sciences, University Politehnica of Bucharest, 060042, Romania and \\ 
                 "Simion Stoilow" Institute of Mathematics, Romania \\ e-mail: {\tt mircea.cimpoeas@upb.ro}}
\footnotetext[2]{Faculty of Applied Sciences, University Politehnica of Bucharest, 060042, Romania \\ e-mail: {\tt neacsu.adrian.gabriel@gmail.com}}

\begin{abstract}
Let $\Covering$ be the category of the category of fuzzy coverings,
and $\Partition$, the category of fuzzy partitions. 
We geometrically construct an isomorphism of categories between $\Partition$ and a full subcategory of $\Covering$,
which can be used to derive bijections between fuzzy partitions and fuzzy coverings with finitely many sets.
Also, we establish an isomorphism between $\Covering[n]$, the category of coverings with $n$ fuzzy sets, 
and a subcategory of $\Partition$, whose objects are partitions with $n$ sets which satisfy certain conditions.


\end{abstract}

\setcounter{section}{0}
\section{Introduction}

Lotfi Zadeh \cite{zadeh} introduced the fuzzy sets in 1965. Fuzzy sets are generalizations of sets. The statement 
that an element is in a fuzzy set can be not only false nor true, but
also something in between. Fuzzy sets play a major role in
many topics in mathematics such as fuzzy topology, fuzzy algebra, fuzzy graph
theory and fuzzy analysis. For a comprehensive introduction we refer the reader to \cite{ne}, \cite{ng}.

A fuzzy \emph{covering} of a set $X$ is a collection of fuzzy sets $(A_i)_{i\in I}$, where $I$ is an arbitrary index set, with the property that
for every element $x\in X$ there exists $i\in I$ such that $A_i(x)=1$, see \cite[Definition 1]{deer}. Note that, for infinite coverings, this definition guarantees for any 
$x\in X$ the existence of a fuzzy set $A_i$ to which $x$ fully belongs. The study of fuzzy coverings, related to fuzzy rough set models, provided good tools for machine 
learning algorithms such as feature and instance selection, see \cite{qz} and \cite{deer2}.

A fuzzy \emph{partition} of a set $X$ is a locally finite family of fuzzy sets $(B_i)_{i\in I}$ such that $\sum_{i\in I}B_i(x)=1$ for all $x\in X$, see \cite[Definition 2]{baet}. 
Of course, when the sets $B_i$ are crisp, we reobtain the well known definiton of partition, from the set theory. Several variants and generalization of the notion of fuzzy partition 
were introduced in literature. We reffer the reader to \cite{aaa} for further details on this subject. If the index set $I$ is finite, we note that a fuzzy partition resemble to a 
stochastic process on a finite state space, see Remark \ref{stoc}.

The use of category theory, in order to study the relations between fuzzy coverings and fuzzy partitions, is a recent approach which
was firstly tackled in \cite{N} and later continued in \cite{N2} by the same author. In \cite[Definition 3.1]{N}, it was introduced the category of fuzzy coverings, denoted $\Covering$, 
whose objects are fuzzy coverings $(X,(A_i)_{i\in I})$ and whose morphisms between two objects $(X,(A_i)_{i\in I})$ and $(Y,(A'_j)_{j\in J})$ are pair of functions $(f,\rho):X\times I\to Y\times J$
such that $A_i(x)\leq A'_{\rho(i)}(f(x))$ for all $x\in X$ and $i\in I$. In \cite[Theorem 4.1]{N}, an isomorphism between $\fCovering$, the category of
fuzzy coverings with finite number of fuzzy sets and a category of fuzzy partitions was provided. However, that isomorphism is not very natural in a sense which will be
explained later on. For the basic prerequisites in category theory we refer the reader to \cite{MacL}.

We introduce the category $\Partition$, see Definition \ref{covdef}, whose objects are fuzzy partitions with finitely many sets, and whose morphisms between two objects 
$(X,(B_i)_{i\in I})$ and $(Y,(B'_j)_{j\in J})$ are pair of functions $(f,\rho):X\times I\to Y\times J$
such that $B_{i}(x) - \bigvee_{i\in I} B_i(x) \leq B'_{\rho(i)}(f(x)) - \bigvee_{j\in J} B'_j (f(x))$ for all $x\in X$ and $i\in I$,
which is a modified version of the category $\Coverage$ from \cite{N}. We explain why this definition is better in Remark \ref{gigi}.

The aim of this paper is to continue this line of research and to find other connections between several subcategories of fuzzy coverings and fuzzy partitions, both with finite number of sets.
We let $\HH=[0,1]^n,\; \FF=\bigcup_{i=1}^n \HH\cap\{x_i=1\}\text{ and }\KK=\HH\cap \{x_1+\cdots+x_n=1\}$.
The central idea of this paper is to interpret the fuzzy coverings with $n$ sets as maps $A:X\to\FF$ and the fuzzy partitions  
with $n$ sets as maps $B:X\to\KK$. In order to establish isomorphisms between subcategories of $\Covering$ and $\Partition$ we will use several results of convex geometry which relates $\FF$ to $\KK$. For a friendly introduction on this topic we refer the reader to \cite{hug}.

The outline of the article is as follows. In Section 2, we recall some basic definitions regarding 
fuzzy sets, fuzzy coverings and fuzzy partitions. Also, we recall the definition of the
$\Covering$ category from \cite{N} and we introduce $\Partition$, the category of fuzzy partitions with finitely many sets.

In Section 3, using the orthogonal projection of $\KK$ into $\FF$, we construct an isomorphism of 
categories between $\Partition$ and a subcategory of finite coverings, see Theorem \ref{t1}. Using this isomorphism,
in Section 4 and Section 5 we establish two bijections between fuzzy partitions and fuzzy coverings with $n$ sets, 
see Theorem \ref{kungfu1} and Theorem \ref{t3}.

In Section 6, we explore other connections between fuzzy coverings and partitions with a fixed number of sets $n$. In Theorem \ref{t2}
we construct a isomorphism of categories between $\Covering[n]$, the category of fuzzy coverings witn $n$ sets, and a subcategory of
partitions with $n$ sets. Consequently, in Section 7, we obtain in Theorem \ref{t222} a new bijection 
between the partitions and the coverings with $n$ sets of a set $X$. In Section 8 and Section 9, we particularize our results 
for $n=2$ and $n=3$.

\section{Preliminaries}

Let $X$ be a nonempty set.

\begin{dfn}
We say that $A$ is a \emph{fuzzy set}, or a \emph{fuzzy subset of $X$}, if $A:X\rightarrow[0,1]$ is a function.
$A(x)$ is the membership degree to which $x$ belongs to $A$.
\end{dfn} 

\begin{dfn}(\cite[Definition 1]{deer})
We say that $(X,(A_{i})_{i \in I})$ is a \emph{fuzzy covering}, or, simply, a \emph{covering} of $X$, 
if $A_{i}:X\rightarrow \left[0,1 \right]$ are fuzzy sets such that for any
$x \in X$, there exists $i \in I$ with $A_{i}(x)=1$. 

In this case, we can also say that $X$ is \emph{covered} by the fuzzy sets $A_{i}$, $i \in I$.
\end{dfn}

\begin{dfn}(\cite[Definition 3.1]{N})
Let $\Covering$ be the category which has:
\begin{enumerate}
\item[(1)]  
$\Ob(\Covering)= \left \{  \left( X,\left( A_{i} \right)_{i \in I} \right)\;:\;\left( X,\left(A_{i} \right)_{i \in I} \right) \text{ is a fuzzy covering} \right \}.$
\item[(2)] $ \Hom\left( \left(X,\left(A_{i}\right)_{i \in I} \right), \left(Y,\left(A'_{j}\right)_{j \in J} \right) \right) = 
            \{ (f,\rho)\;:\;f:X \rightarrow Y, \rho:I \rightarrow J$, \\ 
           $\text{ such that }A_{i}(x)\leq A'_{\rho(i)}(f(x)),$ for all $x \in X$ and $i \in I \}$.
\item[(3)] $(g,\theta) \circ (f,\rho)=(g \circ f, \theta \circ \rho)\in \Hom\left( \left(X,\left(A_{i}\right)_{i \in I} \right), \left(Z,\left(A''_{k}\right)_{k \in K} \right)\right)$, \\
$ \text{ for all }(f,\rho) \in \Hom\left( \left(X,\left(A_{i}\right)_{i \in I} \right), \left(Y,\left(A'_{j}\right)_{j \in J} \right)\right)\text{ and }$ \\
$ (g,\theta) \in \Hom\left( \left(Y,\left(A'_{j}\right)_{j \in J} \right), \left(Z,\left(A''_{k}\right)_{k \in K} \right)\right).$
\item[(4)] The identity morphism, $\id_{\left(X,\left(A_{i}\right)_{i \in I} \right)} =\left( \id_{X}, \id_{I} \right)$, for all $\left(X,\left(A_{i}\right)_{i \in I} \right)\in  \Ob(\Covering).$
\end{enumerate}
\end{dfn}

\begin{dfn}(\cite[Definition 4.3]{N})
We define the category $\fCovering$ as the full subcategory of $\Covering$
which has a finite number of fuzzy sets in their coverings.
\end{dfn}

\begin{dfn}(see \cite[Definition 2]{baet})
Let $I$ be a set of indices. We say that $(X,(B_{i})_{i \in I})$ is a \emph{fuzzy partition}, or, simply, 
a \emph{partition} of $X$, if:
\begin{enumerate}
\item[(1)] $(B_{i})_{i\in I}$ is a locally finite family of fuzzy sets, i.e. for any $x\in X$, the index set $I_x:=\{i\in I\;:\;B_i(x)>0\}$ is finite.
\item[(2)] $\sum_{i\in I}B_i(x)=1$ for all $x\in X$.
\end{enumerate}
\end{dfn}

Note that the condition (1) is automatically fulfilled if the index set $I$ is finite.
We denote $[n]:=\{1,2,\ldots,n\}$, where $n$ is a positive integer.

\begin{obs}\label{stoc}\rm
The first definition of a fuzzy partition with finitely many sets was given by Ruspini in \cite{rusp}, 
in terms of probability theory. We recall that a \emph{stochastic process} is a collection $P = (P_x)_{x\in X}$ of random
variables, all taking values in the same set $S$ (called the state space); see for instance \cite{doob}. 

Assume that $S=[n]$. For each $x\in X$, we have $P_x:\begin{pmatrix} 1 & 2 & \cdots & n \\ p_x(1) & p_x(2) & \cdots & p_x(n) \end{pmatrix}$,
where $p_x(k)=P(P_x=k)\geq 0$ and $p_x(1)+\cdots +p_x(n)=1$. This shows that we can interpret $P$ as a partition of the set $X$ with $n$ subsets
, i.e. $P=(X,(B_i(x))_{i\in [n]})$ where $B_i(x)=p_x(i)$ for all $i\in [n]$ and $x\in X$.
\end{obs}


We introduce the following category, which is a modified version of the $\Coverage$ category introduced in \cite{N}:

\begin{dfn}(compare with \cite[Definition 4.2]{N})\label{covdef}
Let $\Partition$ be the category which has:
\begin{enumerate}
\item[(1)] $\Ob(\Partition)= \left \{  \left( X,\left( B_{i} \right)_{i \in I} \right)\;:\; \left( X,\left(B_{i} \right)_{i \in I} \right) 
            \text{ is a fuzzy partition} \right \}$, where $I$ is a finite set of indices.
\item[(2)] $ \Hom\left( \left(X,\left(B_{i}\right)_{i \in I} \right), \left(Y,\left(B'_{j}\right)_{j \in J} \right) \right) = 
        \{ (f,\rho)\;:\;f:X \rightarrow Y,\;\rho:I \rightarrow J$, such that \\
 $B_{i}(x) - \bigvee_{i\in I} B_i(x) \leq B'_{\rho(i)}(f(x)) - \bigvee_{j\in J} B'_j (f(x))$, for all $x \in X$ and $i \in I \}$.
\item[(3)] $(g,\theta) \circ (f,\rho)=(g \circ f, \theta \circ \rho)\in \Hom\left( \left(X,\left(B_{i}\right)_{i \in I} \right), \left(Z,\left(B''_{k}\right)_{k \in K} \right)\right)$, \\
$ \text{ for all }(f,\rho) \in \Hom\left( \left(X,\left(B_{i}\right)_{i \in I} \right), \left(Y,\left(B'_{j}\right)_{j \in J} \right)\right)\text{ and }$ \\
$ (g,\theta) \in \Hom\left( \left(Y,\left(B'_{j}\right)_{j \in J} \right), \left(Z,\left(B''_{k}\right)_{k \in K} \right)\right).$
\item[(4)] The identity morphism, $\id_{\left(X,\left(B_{i}\right)_{i \in I} \right)} =\left( \id_{X}, \id_{I} \right)$, for all $\left(X,\left(B_{i}\right)_{i \in I} \right)\in  \Ob(\Partition).$
\end{enumerate}
\end{dfn}

\begin{obs}\label{gigi}\rm
Note that if $(X,(A_i)_{i\in I})$ is a covering, then $\bigvee_{i\in I} A_i(x)=1$ for all $x\in X$. Hence, 
$(f,\rho):(X,(A_i)_{i\in I}) \to (Y,(A'_j)_{j\in J})$ is a morphism in $\Covering$ if and only if
$$A_i(x) - \bigvee_{i\in I} A_i(x) \leq A'_{\rho(i)}(f(x)) - \bigvee_{j\in J} A'_j(f(x))\text{ for all }i\in I,x\in X,$$
a condition similar to the condition for morphisms in $\Partition$, given in Definition \ref{covdef}.
\end{obs}

\section{An isomorphism between $\Partition$ and a subcategory of $\Covering$}

Let $n\geq 2$ be an integer. 
Let $\HH:=[0,1]^n\subset \mathbb R^n$ be a $n$-dimensional cube. We also consider:
$$\FF=\bigcup_{i=1}^n (\HH\cap\{x_i=1\})\text{ and }\KK=\HH\cap \{x_1+\cdots+x_n=1\}.$$
Note that the $(n-1)$-dimensional simplex $\KK$ is the convex hull of the points
$(1,0,\ldots,0)$, $(0,1,\ldots,0),\ldots,(0,\ldots,0,1)$.

\begin{prop}\label{p1}
With the above notations, the orthogonal projection of $\KK$ to $\FF$ is the map
$$\Phi:\KK \to \FF,\;\Phi(b)=b+(1-\bigvee_{i\in[n]}b_i)\uu,$$
where $\uu=(1,\ldots,1)$. Moreover, $\Phi$ is injective and the image of $\Phi$ is 
$$\GG:=\Imm(\Phi)=\{(a_1,\ldots,a_n) \in \FF \;: \; \sum_{i\in[n]} a_i \leq na_i+1,\;i\in [n] \}.$$
Also, the inverse of $\Phi$ (with the codomain restricted to $\GG$) is
$$\Phi^{-1}:\GG \to \KK,\;\Phi^{-1}(a)=a+\frac{1}{n}(1-\sum_{i\in[n]} a_i)\uu.$$
\end{prop}

\begin{proof}
We consider the hyperplane $\Pi=\{x_1+\cdots+x_n=1\}$. We denote $\FF_i=\HH\cap \{x_i=1\}$ for all $i\in [n]$.
Let $b\in \KK$ and let $\ell\subset \mathbb R^n$ be the line perpendicular on $\Pi$ which contains $b$.
Then $$\ell:x_1-b_1=x_2-b_2=\cdots = x_n-b_n.$$
Let $i_0\in [n]$ such that $b_{i_0}=\bigvee_{i\in[n]}b_i$.
A straightforward computation shows that
$$\ell\cap \FF = \ell \cap \FF_{i_0} = \{ (1-b_{i_0}+b_1,\ldots,1-b_{i_0}+b_n) \} = \{ \Phi(b) \}.$$
Hence, $\Phi(b)$ has the required expression.

Now, assume that $\Phi(b)=\Phi(b')$ for some $b,b'\in \KK$. It follows that 
\begin{equation}\label{mimi}
b_i - \bigvee_{i\in[n]}b_i = b'_i - \bigvee_{i\in[n]}b'_i\text{ for all }i\in [n].
\end{equation}
In particular, if $b_{i_0}=\bigvee_{i\in[n]}b_i$ then $b'_{i_0}=\bigvee_{i\in[n]}b'_i$.
From \eqref{mimi} it follows that:
\begin{equation}\label{mimisor}
b_i-b_{i_0}=b'_i-b'_{i_0}\text{ for all }i\in [n].
\end{equation}
Since $b_1+\cdots+b_n=b'_1+\cdots+b'_n=1$, from \eqref{mimisor} it follows that $b=b'$. Thus $\Phi$ is injective.

Let $a=(a_1,\ldots,a_n)\in \Imm(\Phi)$. In order to determine $b\in\KK$ such that $\Phi(b)=a$, we
let $\ell$ be the line perpendicular on $\Pi$ which contains $a$. Then
$$\ell:x_1-a_1=x_2-a_2=\cdots=x_n-a_n.$$
It follows that $\ell\cap\Pi=\{b\}$ and moreover
$$b_1-a_1=\cdots=b_n-a_n\text{ and }b_1+\cdots+b_n=1.$$
We have that
\begin{equation}\label{summ}
(b_1-a_1)+\cdots+(b_n-a_n) = (b_1+\cdots+b_n) - (a_1+\cdots+a_n) = 1 - (a_1+\cdots+a_n).
\end{equation}
On the other hand, since $b_1-a_1=\cdots=b_n-a_n$, \eqref{summ} implies
$$ n(b_i - a_i) = 1 - \sum_{i\in[n]} a_i,\;\text{ for all }i\in [n],$$
which is equivalent to 
\begin{equation}\label{ineg}
 nb_i = na_i  - \sum_{i\in[n]} a_i+1,\;\text{ for all }i\in [n].
\end{equation}
The condition $b\in \KK$ is equivalent to $b\in \Pi$ and $b_i\geq 0$ for all $i\in [n]$.

From \eqref{ineg}, $b_i\geq 0$ is equivalent to $a_1+\cdots+a_n \leq na_i+1$. Therefore $b\in \KK$ is
equivalent to $a\in\GG$, as required. The expression of $\Phi^{-1}$ follows also from \eqref{ineg}.
\end{proof}

Note that, if $a\in \GG$ and $b\in \KK$ such that $\Phi(b)=a$, then
\begin{equation}\label{cucu}
 a_i=b_i+1-\bigvee_{i\in [n]}b_i\text{ for all }i\in [n].
\end{equation}

For convenience, from now on, we will assume that the finite sets of indices $I$ and $J$ which appear in partitions and finite coverings 
are of the form $I=[n]$ and $J=[m]$, where $n$ and $m$ are positive integers. This assumption does not impede on the generality,
since any nonempty finite set can be put in bijection with a set of the form $[n]$, where $n$ is a positive integer.

\begin{dfn}\label{green}
A covering $(X,(A_i)_{i\in [n]})$ of $X$ is called a \emph{good covering} if it satisfies the conditions
$$ \sum_{i\in[n]}A_i(x) \leq n\cdot A_i(x)+1,\text{ for all }i\in [n]\text{ and }x\in X.$$
Equivalently, $(X,(A_i)_{i\in [n]})$ is a good covering if and only if $(A_1(x),\ldots,A_n(x))\in\GG$ for any $x\in X$.

We define the category $\gCovering$ as the full subcategory of $\Covering$ whose objects are good coverings.
\end{dfn}

\begin{obs}\rm
(1) $(X,(A_i)_{i\in [n]})$ is a good covering $\Leftrightarrow \sum_{i\in[n]}A_i(x) \leq n\cdot \bigwedge_{i \in [n]} A_i(x)+1.$

(2) Let $(X,(A_i)_{i\in [n]})$ and $(Y,(A'_j)_{j\in [m]})$ be two good coverings. Then the direct product 
    $(X\times Y, (A_i\wedge A'_j)_{(i,j)\in [n]\times[m]})$ is not necessarily a good covering. Indeed, assume $n=m=3$ and $x\in X$, 
		$y\in Y$ such that $A_1(x)=0.5$, $A_2(x)=1$, $A_3(x)=1$, $A'_1(y)=0.5$, $A'_2(y)=1$ and $A'_3(y)=1$. It is clear that
		$$A_1(x)+A_2(x)+A_3(x) = 3A_1(x)+1 = 2.5\text{ and }A'_1(y)+A'_2(y)+A'_3(y) = 3A'_1(y)+1 = 2.5.$$
		On the other hand, we have that $$\sum_{i,j\in[3]}A_i(x)\wedge A'_j(y) = 5\cdot 0.5 + 4\cdot 1 =6.5 > 5.5 = 9\cdot 0.5 + 1 = 9A_1(x)\wedge A'_1(y)+1.$$
\end{obs}

\begin{dfn}
We define the functor $\Fu:\Partition \to \gCovering$ as follows:
\begin{enumerate}
\item[(1)] For any partition $(X,(B_i)_{i\in[n]})$, we let $\Fu((X,(B_i)_{i\in[n]})):=(X,(A_i)_{i\in[n]})$,
      where $$A_i(x)=B_i(x)+1-\bigvee_{i\in[n]} B_i(x),\text{ for all }i\in[n],x\in X.$$
\item[(2)] For any morphism $(f,\rho):(X,(B_i)_{i\in[n]}) \to (Y,(B'_j)_{j\in[m]})$ we let $\Fu((f,\rho))=(f,\rho)$.
\end{enumerate}
We also define the functor $\Gu:\gCovering \to \Partition$ as follows:
\begin{enumerate}
\item[(1)] For any good covering $(X,(A_i)_{i\in[n]})$, we let $\Gu((X,(A_i)_{i\in[n]})):=(X,(B_i)_{i\in[n]})$,
      where $$B_i(x)=A_i(x)-\frac{1}{n}\sum_{i\in[n]} A_i(x) + \frac{1}{n},\text{ for all }i\in[n],x\in X.$$
\item[(2)] For any morphism $(f,\rho):(X,(B_i)_{i\in[n]}) \to (Y,(B'_j)_{j\in[m]})$ we let $\Gu((f,\rho))=(f,\rho)$.
\end{enumerate}
\end{dfn}

\begin{obs}\rm
Note that $(X,(A_i)_{i\in [n]})$ is a covering if and only if there exists a map $f:X\to \mathcal F$ such that
$f(x)=(A_1(x),\ldots,A_n(x))$ for all $x\in X$. In particular, $(X,(A_i)_{i\in [n]})$ is a good covering if and only
if $\Imm(f)\subset \mathcal G$. Also, $(X,(B_i)_{i\in [n]})$ is a partition if and only if there exists a map $g:X\to\mathcal K$
such that $g(x)=(B_1(x),\ldots,B_n(x))$ for all $x\in X$.
\end{obs}

\begin{prop}\label{pp1}
With the above notations, the functors $\Fu$ and $\Gu$ are well defined and fully faithful.
\end{prop}

\begin{proof}
Let $(X,(B_i)_{i\in[n]})$ be a partition and let $x\in X$. 
For convenience, we denote $b_i:=B_i(x)$ and $a_i:=A_i(x)$ for all $i\in [n]$.
Let $a=(a_1,\ldots,a_n)$ and $b=(b_1,\ldots,b_n)$. From the definition of $\Fu$, we note that
$a=\Phi(b)$, where $\Phi$ is the map defined in Proposition \ref{p1}.
It follows that $a\in \GG$. Since $x$ was arbitrary, it follows that 
$(X,(A_i)_{i\in[n]})$ is a good covering.

Let $\Id_{(X,(B_i)_{i\in[n]})} =(\Id_X,\Id_{[n]})$ be the identity morphism of $(X,(B_i)_{i\in[n]})$ in $\Partition$.
Then $\Fu((\Id_X,\Id_{[n]}))=(\Id_X,\Id_{[n]})$ is the identity morphism of $(X,(A_i)_{i\in[n]})$ in $\gCovering$.

Now, let $(f,\rho):(X,(B_i)_{i\in[n]}) \to (Y,(B'_j)_{j\in[m]})$ be a morphism in $\Partition$, that is
$f:X\to Y$ and $\rho:[n]\to[m]$ are functions such that 
$$B_{i}(x) - \bigvee_{i\in [n]} B_i(x) \leq B'_{\rho(i)}(f(x)) - \bigvee_{j\in [m]} B'_j (f(x)),\text{ for all }x \in X\text{ and }i \in [n].$$
Assume that $\Fu((X,(B_i)_{i\in[n]}))=(X,(A_i)_{i\in [n]})$ and $\Fu((X,(B'_j)_{j\in[m]}))=(X,(A'_j)_{j\in [m]})$.

Since $(f,\rho):(X,(B_i)_{i\in[n]}) \to (Y,(B'_j)_{j\in[m]})$ is
a morphism in $\Partition$, for any $x\in X$ and $i\in [n]$ we have that
\begin{equation}\label{kkoi}
 A_i(x)=B_{i}(x)+1 - \bigvee_{i\in [n]} B_i(x) \leq B'_{\rho(i)}(f(x))+1 - \bigvee_{j\in [m]} B'_j (f(x)) = A'_{\rho(i)}(f(x)),
\end{equation}
hence $\Fu((f,\rho))=(f,\rho):(X,(A_i)_{i\in [n]}) \to (X,(A'_j)_{j\in [m]})$ is a morphism in $\gCovering$.

Conversely, if $(f,\rho):(X,(A_i)_{i\in [n]}) \to (X,(A'_j)_{j\in [m]})$ is a morphism in $\gCovering$, from \eqref{kkoi} it follows
that $(f,\rho):(X,(B_i)_{i\in[n]}) \to (Y,(B'_j)_{j\in[m]})$ is
a morphism in $\Partition$. Therefore, we have:
$$\Hom_{\Partition}((X,(B_i)_{i\in[n]}),(Y,(B'_j)_{j\in[m]}) ) =\Hom_{\gCovering}((X,(A_i)_{i\in[n]}),(Y,(A'_j)_{j\in[m]}) ).$$
Now, since $\Fu((f,\rho))=(f,\rho)$, it is clear that $\Fu$ is fully faithful.

The proof of the fact that $\Gu$ is well defined and fully faithful is similar so we will omit it.
\end{proof}

\pagebreak 

\begin{teor}\label{t1}
The categories $\Partition$ and $\gCovering$ are isomorphic.
\end{teor}

\begin{proof}
The functors $\Fu$ and $\Gu$, according to Proposition \ref{pp1}, are well defined and fully faithful.
Hence, it is enough to prove that i) $\Gu\circ\Fu = \Id_{\Ob(\Partition)}$ and ii) $\Fu\circ\Gu = \Id_{\Ob(\gCovering)}$.

i) Let $(X,(B_i)_{i\in [n]})$ a partition and let $(X,(A_i)_{i\in[n]})=\Fu((X,(B_i)_{i\in [n]}))$.
   As in the proof of Proposition \ref{p1}, we fix $x\in X$ and denote $b_i=B_i(x)$ and $a_i=A_i(x)$ for all $i\in [n]$.
	 Let $a=(a_1,\ldots,a_n)$ and $b=(b_1,\ldots,b_n)$. 
	
	 From the definition of $\Fu$ and Proposition \ref{p1} we have $a=\Phi(b)$.
	 On the other hand, if we let $(X,(B'_i)_{i\in [n]})=\Gu( (X,(A_i)_{i\in[n]}))$ and we denote $b'_i=B'_i(x)$ for $i\in [n]$
	 and $b'=(b'_1,\ldots,b'_n)$, then, from Proposition \ref{t1}, it follows that $b'=\Phi^{-1}(a)$. Therefore $b'=\Phi^{-1}(\Phi(b))=b$,
	 which completes the proof.
	
ii) The proof is similar.
\end{proof}

A natural question is the following: Is there a natural bijective correspondence between good coverings and finite coverings? 
We will see in the following sections that the answer is yes and we will present two ways of establishing this.
Suppose we did find such a correspondence, using the isomorphism of categories given in Theorem \ref{t1}, then we can deduce a bijective
correspondence between partitions and finite coverings. Unfortunately, there is an unavoidable drawback: these associations are not categorial.

\section{A bijection between partitions and finite coverings}

First, we need the following lemma:

\begin{lema}\label{lemuita}
Let $\varphi_1:\FF\to\KK$,  $\varphi_1(c):=\frac{1}{\sum_{i\in[n]} c_i}c$. Then $\varphi_1$ is bijective and, moreover,
we have $\varphi_1^{-1}:\KK\to \FF$, $\varphi_1^{-1}(b) = \frac{1}{\bigvee_{i\in [n]}b_i}b$.
\end{lema}

\begin{proof}
Let $c\in \FF$. We consider the line $\ell$ determined by $c$ and the origin $O$. Then $\ell$ intersects $\KK$ in $b$.
We claim that $b=\varphi_1(c)$. Indeed, since $\ell\cap \KK =\{b\}$ it follows that $b=\alpha c$ for some $\alpha\in\mathbb R$ and $b_1+\cdots+b_n=1$. This implies $\alpha=\sum_{i\in[n]} c_i$. See also \cite[Remark 4.5]{N}.

Conversely, $c=\varphi_1^{-1}(b)$ is the intersection point of the line determined by $b$ and $O$ with $\FF$. 
Hence $c=\beta b$ for some $\beta\in\mathbb R$. Note that $c\in \FF$ if and only if $\bigvee_{i\in [n]}c_i=1$.
From $c=\beta b$ it follows that $\beta=\frac{1}{\bigvee_{i\in [n]}b_i}$, as required. 
\end{proof}

\begin{dfn}
If $(X,(C_i)_{i\in[n]})$ is a covering, then we let $\Fu_1((X,(C_i)_{i\in[n]})):=(X,(B_i)_{i\in [n]})$, where
$$(B_1(x),\ldots,B_n(x))=\varphi(C_1(x),\ldots,C_n(x)) = \left( \frac{C_1(x)}{\sum_{i\in[n]}C_i(x)},\ldots,\frac{C_n(x)}{\sum_{i\in[n]}C_i(x) } \right).$$
If $(X,(B_i)_{i\in[n]})$ is a partition then we let $\Gu_1((X,(B_i)_{i\in[n]})):=(X,(C_i)_{i\in [n]})$, where
$$(C_1(x),\ldots,C_n(x))=\varphi_1^{-1}(B_1(x),\ldots,B_n(x)) = 
\left( \frac{B_1(x)}{\bigvee_{i\in[n]}B_i(x)},\ldots, \frac{B_n(x)}{\bigvee_{i\in[n]}B_i(x)}\right) .$$
\end{dfn}

\begin{teor}\label{kungfu1}
With the above notations, the maps $\Fu_1:\Ob(\fCovering)\to\Ob(\Partition)$ and $\Gu_1:\Ob(\Partition)\to\Ob(\fCovering)$
are bijective and inverse to each other. 
\end{teor}

\begin{proof}
It is a direct consequence of Lemma \ref{lemuita}.
\end{proof}

\begin{obs}\rm
Note that $\Fu_1$ and $\Gu_1$ are defined similarly to the way the functors $F$ and $G$ 
from \cite[Theorem 4.1]{N} are defined on objects. However, $\Fu_1$ and $\Gu_1$ cannot be extended 
to functors in a natural way. For instance, if we set $\Fu_1((f,\rho))=(f,\rho)$ and $\Gu_1((f,\rho))=(f,\rho)$, where 
$(f,\rho): (X,(C_i)_{i\in [n]}) \to (X',(C'_j)_{j\in [m]})$ is a morphism, then $C_i(x)\leq C'_{\rho(i)}(f(x))$ for all $x\in X$ and $i\in [n]$.

If we let $(X,(B_i)_{i\in [n]})=\Fu_1(X,(C_i)_{i\in [n]})$ and 
$(X,(B'_j)_{j\in [m]})=\Fu_1(X,(C'_j)_{j\in [m]})$, then the inequality 
$$B_{i}(x) - \bigvee_{i\in [n]} B_i(x) =  \frac{C_i(x)}{\sum_{i\in [n]}C_i(x) } \leq
\frac{C'_{\rho(i)}(f(x))}{\sum_{j\in [m]}C'_j(f(x)) }=
 B'_{\rho(i)}(f(x)) - \bigvee_{j\in [m]} B'_j (f(x)),$$
may not hold. Similarly, if $(f,\rho):(X,(B_i)_{i\in [n]}) \to (X',(B'_j)_{j\in [m]})$ is a morphism in $\Partition$, then there is no guarantee for 
$\Gu_1((f,\rho))=(f,\rho):(X,(C_i)_{i\in [n]}) \to (X',(C'_j)_{j\in [m]})$ to be a morphism in $\fCovering$.
\end{obs}


\begin{prop}\label{p22}
The map $\Phi_1:\FF \to \GG$, $\Phi_1:=\Phi\circ \varphi_1$, is a bijection and its inverse is 
$\Phi_1^{-1}:\GG \to \FF$, $\Phi_1^{-1}=\varphi_1^{-1}\circ \Phi^{-1}$. Moreover, we have that:
\begin{enumerate}
 \item[(1)] $\Phi_1(c_1,\ldots,c_n)= \left( \frac{c_1 - 1}{\sum_{i\in[n]} c_i} + 1 , \ldots, \frac{c_n - 1}{\sum_{i\in[n]} c_i} + 1  \right)$.
 \item[(2)] $\Phi_1^{-1}(a_1,\ldots,a_n) = \left( \frac{na_1-n }{ n + 1 - \sum_{i\in[n]} a_i } + 1,\ldots, \frac{na_n -n}{ n + 1 - \sum_{i\in[n]} a_i }+1  \right)$.
\end{enumerate}
\end{prop}

\begin{proof}
 It follows from Proposition \ref{p1} and Lemma \ref{lemuita}  by straightforward computations.
\end{proof}

\begin{dfn}
If $(X,(C_i)_{i\in [n]})$ is a covering of $X$, we let $\overline{F}((X,(C_i)_{i\in [n]})):=(X,(A_i)_{i\in [n]})$, where
 $$(A_1(x),\ldots,A_n(x)):
   = \left( \frac{C_1(x) - 1}{\sum_{i\in[n]} C_i(x)} + 1 , \ldots, 
   \frac{C_n(x) - 1}{\sum_{i\in[n]} C_i(x)} + 1  \right)\text{, for all }x\in X.$$
If $(X,(A_i)_{i\in[n]})$ is a good covering of $X$, we let $\overline{G}((X,(A_i)_{i\in [n]})):=(X,(C_i)_{i\in [n]})$, where
  $$(C_1(x),\ldots,C_n(x)):
	  =\left( \frac{nA_1(x) -n }{ n + 1 - \sum_{i\in[n]} A_i(x)  }+1,\ldots, 
	  \frac{nA_n(x) - n }{ n + 1 - \sum_{i\in[n]} A_i(x) } +1  \right)\text{, for all }x\in X.$$
\end{dfn}

\begin{cor}\label{c2}
With the above notations, $\overline{F}:\Ob(\fCovering)\to\Ob(\gCovering)$ and 
$\overline{G}:\Ob(\gCovering)\to\Ob(\fCovering)$ are bijective and inverse to each other.
\end{cor}

\begin{proof}
 The conclusion follows from Proposition \ref{p22}.
\end{proof}

\begin{obs}\rm
Note that the bijections given in Corollary \ref{c2} cannot be extended to isomorphisms of categories.
\end{obs}

\section{A second bijection between partitions and finite coverings}

In the following, we construct another bijection between partitions and finite coverings. First, we fix some notations:

We denote $S_n$ the set of permutations of $[n]$. Let $\sigma\in S_n$. We consider the set:
$$\FF_{\sigma}=\{c\in\FF\;:\;1=c_{\sigma(1)}\geq c_{\sigma(2)} \geq \cdots \geq c_{\sigma(n)}\}.$$
Also, we define $\GG_{\sigma}=\{a\in\FF\;:\;1=a_{\sigma(1)}\geq a_{\sigma(2)} \geq \cdots \geq a_{\sigma(n)}\}.$
Note that 
$$\FF = \bigcup_{\sigma\in S_n} \FF_{\sigma} \text{ and } \GG = \bigcup_{\sigma\in S_n} \GG_{\sigma}.$$
For a vector $(x_1,\ldots,x_n)\in \mathbb R^n$, we let $(x_1,\ldots,x_n)^{\sigma}:=(x_{\sigma(1)},\ldots,x_{\sigma(n)})$.

With the above notations, we have the following:

\begin{lema}\label{fis}
Let $\sigma\in S_n$. The map $\Phi_{2,\sigma}:\GG_{\sigma} \to \FF_{\sigma}$, defined by \small
$$ \Phi_{2,\sigma}(a_1,a_2,\ldots,a_n):=
(a_{\sigma(1)},a_{\sigma(2)},2a_{\sigma(3)}-a_{\sigma(2)},3a_{\sigma(4)}-a_{\sigma(2)}-a_{\sigma(3)},\ldots,
(n-1)a_{\sigma(n)}-a_{\sigma(2)}-\cdots-a_{\sigma(n-1)})^{\sigma^{-1}},$$ \normalsize
is bijective and its inverse is $\Phi_{2,\sigma}^{-1}:\FF_{\sigma}\to \GG_{\sigma}$, defined by
$$\Phi_{2,\sigma}^{-1}(c_1,\ldots,c_n)=(a_{\sigma(1)},\ldots,a_{\sigma(n)})^{\sigma^{-1}},$$
where $a_{\sigma(1)}=c_{\sigma(1)}$, $a_{\sigma(2)}=c_{\sigma(2)}$ and 
$$a_{\sigma(k)} = \frac{1}{1\cdot 2}c_{\sigma(2)} + \frac{1}{2\cdot 3}c_{\sigma(3)} + \cdots + \frac{1}{(k-2)(k-1)}c_{\sigma(k-1)}+\frac{1}{k-1}c_{\sigma(k)},$$
for all $3\leq k\leq n$.
\end{lema}

\begin{proof}
Without any loss of generality, we can assume that $\sigma=e\in S_n$ is the identical permutation. 
For $a\in \GG_e$, we have that
\begin{equation}\label{vare}
\Phi_{2,e}(a_1,a_2,\ldots,a_n)=(a_1,a_2,2a_3-a_2,3a_4-a_2-a_3,\ldots,(n-1)a_n-a_2-\cdots-a_{n-1}).
\end{equation}
Also 
$\GG_e=\{(a_1,a_2,\ldots,a_n)\;:\;1=a_1\geq a_2\geq \cdots \geq a_n \geq 0,\;a_2+\cdots+a_{n-1}\leq (n-1)a_n\}.$
Since $1=a_1\geq a_2\geq \cdots \geq a_n$, be straightforward computations it follows that
$$1\geq a_2\geq 2a_3-a_2 \geq \cdots \geq (n-1)a_n-a_2-\cdots-a_{n-1}.$$
On the other hand, since $a\in \GG_e$, it follows that $(n-1)a_n-a_2-\cdots-a_{n-1}\geq 0$. This shows that
the function $\Phi_{2,e}$ is well defined. 

Let $c=(c_1,\ldots,c_n)\in\FF_e$. We claim that there exists a unique $a=(a_1,\ldots,a_n)\in \GG_e$ such that $\Phi_{2,e}(a)=c$
and hence we can define $\Phi_{2,e}^{-1}(c):=a$. Indeed, from \eqref{vare}, $\Phi_{2,e}(a)=c$ is equivalent to the
linear system 
\begin{equation}\label{sis}
\begin{cases} a_1 = c_1 \\ a_2=c_2 \\ 2a_3-a_2 = c_3 \\ \vdots \\ (n-1)a_n - a_2-\cdots-a_{n-1} =c_n  \end{cases},
\end{equation}
which has the associated determinant $\Delta:=\begin{vmatrix} 1 & 0 & 0 & 0 & \cdots & 0 \\ 0 & 1 & 0 & 0 & \cdots & 0 \\ 0 & -1 & 2 & 0 & \cdots & 0 \\ 0 & -1 & -1 & 3 & \cdots & 0 
\\ \vdots & \vdots & \vdots & \vdots & \ddots & \vdots  \\ 0 & -1 & -1 & -1 & \cdots & n-1 \end{vmatrix} = (n-1)! \neq 0$. It follows that
$a_1,\ldots,a_n$ are uniquely determined from $c_1,\ldots,c_n$,
using the Cramer's rule. Moreover, since $c_1\geq c_2\geq \cdots \geq c_n$ it follows easily from \eqref{sis}
that $a_1\geq a_2\geq a_3 \geq \cdots \geq a_n$.
Also, $0\leq c_n = (n-1)a_n-a_2-\cdots-a_{n-1}$, thus $a\in\GG_e$, as required.

We let $s_k:=a_2+\cdots+a_k$, $2\leq k\leq n$, and $s_1:=0$. We also let $c'_k:=\frac{1}{k-1}c_k$, $2\leq k\leq n$. From \eqref{sis} it follows that
\begin{equation}\label{indu}
 s_2=c'_2,\; s_k=\frac{k}{k-1}s_{k-1}+c'_k,\;3\leq k\leq n.
\end{equation}
Using induction on $k\geq 2$, from \eqref{indu} it follows that:
\begin{equation}\label{esn}
 s_k = \frac{k}{2}c'_2+\frac{k}{3}c'_3+\cdots+\frac{k}{k}c'_k,\;2\leq k\leq n.
\end{equation}
Obviously, $a_1=c_1$ and $a_2=c_2$. From \eqref{esn} it follows that 
$$a_k = s_k-s_{k-1} = \frac{1}{1\cdot 2}c_2 + \frac{1}{2\cdot 3}c_3 + \cdots + \frac{1}{(k-2)(k-1)}c_{k-1}+\frac{1}{k-1}c_k,$$
for all $3\leq k\leq n$. Thus, we get the required formula for $\Phi_{2,e}^{-1}$.
\end{proof}

\begin{lema}\label{sum} 
Let $a \in \GG_{\sigma}$, $c=\Phi_{2,\sigma}(a)$ and $1\leq i < j\leq n$. Then $a_i=a_j$ if and only if $c_i=c_j$.
\end{lema}

\begin{proof}
Without any loss of generality we can assume that $\sigma=e$ and $a \in \GG_{e}$. Also, we can assume that $j=i+1$, 
since $a_{i}\geq a_{i+1}$. If $i=1$ then $c_1=a_1$ and $c_2=a_2$, hence there is nothing to prove. 
Now, suppose $i\geq 2$.

For $2 \leq m \leq n$, we let $s_{m}=\sum_{i=1}^m a_{i}$. Also, we let $s_1=0$. 
Then $c_{i}=(i-1)a_{i}-s_{i-1}$ and $c_{i+1}=ia_{i+1}-s_{i}$. It follows that $c_{i+1}-c_{i}=i(a_{i+1} - a_{i})$,
 which completes the proof.
\end{proof}

\begin{prop}\label{p2}
The map $\Phi_2:\GG \to \FF$, $\Phi_2(a):=\Phi_{2,\sigma}(a)$ for $a\in \GG_{\sigma}$, is well defined and bijective.
\end{prop}

\begin{proof}
In order to prove the assertion, it suffices to show that if $a\in \GG_{\sigma}\cap\GG_{\tau}$, where $\sigma,\tau\in S_n$ are
two different permutations, then $\Phi_{2,\sigma}(a)=\Phi_{2,\tau}(a)$.

Since $a \in \GG_{\sigma} \cap \GG_{\tau}$ it follows that there exist $1 \leq \ell_{1} < \ell_{2} < \cdots < \ell_{r} = n$ such that
\begin{equation}\label{permu1}
1=a_{\sigma(1)} = \cdots = a_{\sigma(\ell_1)} > a_{\sigma(\ell_1+1)} = \cdots = a_{\sigma(\ell_2)} > \cdots> a_{\sigma(\ell_{r-1}+1)} = \cdots = a_{\sigma(\ell_r)}.
\end{equation}
and similarly
\begin{equation}\label{permu2}
1=a_{\tau(1)} = \cdots = a_{\tau(\ell_1)} > a_{\tau(\ell_1+1)} = \cdots = a_{\tau(\ell_2)} > \cdots > a_{\tau(\ell_{r-1}+1)} = \cdots = a_{\tau(\ell_r)}.
\end{equation}
Let $\alpha_j=a_{\sigma(j)}=a_{\tau(j)}$ for all $1\leq j\leq n$ and $\alpha=(a_1,\ldots,a_n)$.
From \eqref{permu1} and \eqref{permu2} it follows that 
$$ c:=\Phi_{2,\sigma}(a)=\Phi_{2,e}(\alpha)^{\sigma^{-1}}\text{ and }c':=\Phi_{2,\tau}(a)=\Phi_{2,e}(\alpha)^{\tau^{-1}}.$$
Moreover, \eqref{permu1} and \eqref{permu2} imply that 
$$ \{\sigma(\ell_{j-1}+1),\ldots,\sigma(\ell_j)\} = \{ \tau(\ell_{j-1}+1),\ldots,\tau(\ell_j)\}\text{ for all }1\leq j\leq r,$$
where $\ell_0=0$. From Lemma \ref{sum} and \eqref{permu1} it follows that:
$ c_{\sigma(j)}=c_{\sigma(\ell_i)}\text{ for }\ell_{i-1}<j\leq \ell_i.$

Similarly, from Lemma \ref{sum} and \eqref{permu2} it follows that:
$ c'_{\tau(j)}=c'_{\tau(\ell_i)}\text{ for }\ell_{i-1}<j\leq \ell_i.$

From the above considerations, it is easy to see that $c=c'$, as required.

The inverse of $\Phi_2$ is defined by setting $\Phi_2^{-1}(c):=\Phi_{2,\sigma}^{-1}(c)$, where $c\in\FF_{\sigma}$. 
\end{proof}

\begin{dfn}
Let $(X,(A_i)_{i\in[n]})$ be a good covering of $X$. We define $\overline{F}_2((X,(A_i)_{i\in[n]})):=(X,(C_i)_{i\in[n]})$,
            where:
  $$(C_1(x),\ldots,C_n(x)):=\Phi_2(A_1(x),\ldots,A_n(x)),\text{ for all }x\in X.$$
Let $(X,(C_i)_{i\in [n]})$ be a covering of $X$. We define $\overline{G}_2((X,(C_i)_{i\in[n]})):=(X,(A_i)_{i\in[n]})$,
where 
$$(A_1(x),\ldots,A_n(x)):=\Phi_2^{-1}(C_1(x),\ldots,C_n(x)),\text{ for all }x\in X.$$ 	
\end{dfn}

\begin{prop}\label{c22}
The maps $\overline{F}_2:\Ob(\gCovering)\to \Ob(\fCovering)$ and 
$\overline{G}_2:\Ob(\fCovering)\to \Ob(\gCovering)$ are well defined, bijective and inverse to each other.
\end{prop}

\begin{proof}
The proof is similar to the proof of Corollary \ref{c2}.
\end{proof}

\begin{teor}\label{t3}
The maps $\Fu_2:\Ob(\fCovering)\to\Ob(\Partition)$, $\Fu_2=\Gu \circ \overline{F}_2$, and 
$\Gu_2:\Ob(\Partition)\to\Ob(\fCovering)$, $\Gu_2=\overline{G}_2\circ \Fu$, are bijective and inverse to each other. 
\end{teor}

\begin{proof}
It follows from Theorem \ref{t1} and Proposition \ref{c22}.
\end{proof}

\begin{obs}\rm
Similarly to $\Fu_1$ and $\Gu_1$, the maps $\Fu_2$ and $\Gu_2$ cannot be extended to isomorphisms of categories,
since $\overline{F}_2$ and $\overline{G}_2$ cannot be extended to isomorphisms of categories.
\end{obs}

\section{An isomorphism between the category of coverings with $n$ sets and a subcategory of $\Partition$}

We use the notations from Section $3$. Throughout this section, $n\geq 2$ is a fixed integer. 

Let $a\in \HH$ and let $\ell\subset \mathbb R^n$ be the line perpendicular on $\Pi$ which contains $b$.
As in the proof of Proposition \ref{p1}, it follows that the orthogonal projection of $\HH$ to $\Pi$ is the map
\begin{equation}\label{psi0}
\Psi_0:\HH \to \Pi, \;\Psi_0(a)=a + \frac{1}{n} (1-\sum_{i\in[n]} a_i)\uu,
\end{equation}
where $\uu=(1,\ldots,1)$. Note that each component of the vector $\Psi_0(a)$ 
is larger or equal to $-\frac{n-2}{n}$ and smaller or equal to $1$.

We let $\PP:=\Imm(\Psi_0)$. Since $\HH$ is a convex set and $\Psi_0$ is an orthogonal projection, 
it follows that $\PP\subset \Pi$ is also convex; see \cite{hug} for further details.
 
Moreover, since $\HH$ is the convex hull of $\AAA:=\{(\alpha_1,\ldots,\alpha_n)\;:\;\alpha_i\in\{0,1\}\},$
from \eqref{psi0} it follows that $\PP$ is the convex hull of 
\begin{equation}\label{psi0A}
\Psi_0(\AAA)=\left\{ \left(\frac{1-s}{n}+\alpha_1,\ldots,\frac{1-s}{n}+\alpha_n \right)\;:\;\alpha_i\in\{0,1\},
\;s=\alpha_1+\cdots+\alpha_n \right\}.
\end{equation}

\begin{prop}\label{p3}
We have that $\PP=\Psi_0(\FF)$ and the restricted map $\Psi_0|_{\FF}:\FF \to \PP$ is bijective. Moreover, its
inverse $\Psi_0|_{\PP}^{-1}:\PP \to \FF$ has the expression:
$$\Psi_0|_{\PP}^{-1}(b) = b + (1-\bigvee_{i\in[n]}b_i)\uu,\text{ for all }b\in\PP.$$
\end{prop}

\begin{proof}
Given a point $a\in \HH$, the line orthogonal to $\Pi$ which
contains $a$, intersects $\FF$ in a point $a'$. Since $\Psi_0(a')=\Psi_0(a)$, it follows that $\PP=\Psi_0(\FF)$.
The fact that $\Psi_0|_{\FF}$ is bijective is an easy exercise. The expression of $\Psi_0|_{\PP}^{-1}$ follows from 
Proposition \ref{p1}.
\end{proof}

We recall the following well known result of convex geometry (see \cite{hug}):

\begin{lema}\label{lem1}
We let $\Gg:=\frac{1}{n}\uu = \left(\frac{1}{n},\ldots,\frac{1}{n}\right)$. 
Let $c_{\varepsilon}:\Pi \to \Pi$ be the homothety of factor $\varepsilon>0$ of $\Pi$ with
the center $g$, that is 
\begin{equation}\label{ceps}
c_{\varepsilon}(b) =  \varepsilon b + (1-\varepsilon)\Gg,\;\text{ for all }b\in \Pi.
\end{equation}
Then $c_{\varepsilon}$ is convex and bijective and its inverse is  $c_{\varepsilon}^{-1}=c_{\frac{1}{\varepsilon}}$.
\end{lema}

\begin{prop}\label{p4}
With the above notations, we have:
\begin{enumerate}
\item[(1)] The map 
$\Psi:\FF \to \KK,\;\Psi=c_{\frac{1}{n-1}}\circ \Psi_0|_{\FF},$
is well defined and injective.
\item[(2)] The image of $\Psi$ is the $\overline{\PP}:=$ the convex hull of \small
$$\overline{\AAA}:=\frac{1}{n-1}\cdot \left\{\left(\frac{n-1-s}{n}+\alpha_1,\ldots,\frac{n-1-s}{n}+\alpha_n \right)\;:\; 
\alpha_i\in\{0,1\},\;s=\alpha_1+\cdots+\alpha_n\right\}$$ \normalsize
\item[(3)] The inverse of $\Psi:\FF\to\overline{\PP}$ is 
$$\Psi^{-1}: \overline{\PP} \to \FF, \; \Psi^{-1}=\Psi_0|_{\FF}^{-1}\circ c_{n-1}|_{\overline{\PP}}.$$
\item[(4)] $\Psi(a)=\frac{1}{n-1}a - \frac{1}{n(n-1)}\left( \sum_{i\in[n]} a_i\right)\uu + \frac{1}{n}\uu$ for all $a\in\FF$.
\item[(5)] $\Psi^{-1}(b)=(n-1)(b-\bigvee_{i\in[n]}b_i)+\uu$ for all $b\in\overline{\PP}$.
\end{enumerate}
\end{prop}

\begin{proof}
(1) Since, by Proposition \ref{p3}, $\Psi_0|_{\FF}$ is bijective and, by its definition, $c_{\frac{1}{n-1}}$
is injective, it follows that $\Psi$ is injective.

(2) Since $\Psi_0|_{\HH}$ is bijective, it follows that $\Imm(\Psi)=c_{\frac{1}{n-1}}(\PP)$. The conclusion follows
from the fact that $c_{\frac{1}{n-1}}$ is a convex function and straightforward computations.

(3) It follows from Lemma \ref{lem1}.

(4) Let $a\in \HH$. From \eqref{psi0} and \eqref{ceps}, it follows that
$$\Psi(a)=c_{\frac{1}{n-1}}\left(a + \frac{1}{n} (1-\sum_{i\in[n]} a_i)\uu\right) =  
\frac{1}{n-1}\left(a + \frac{1}{n} (1-\sum_{i\in[n]} a_i)\uu\right) + \frac{n-2}{n(n-1)}\uu.$$

(5) It follows by straightforward computations from (3) and Proposition \ref{p3}.
\end{proof}

Let $c'_{\varepsilon}:\mathbb R^n \to \mathbb R^n$ be the homothety of factor $\varepsilon>0$ of $\mathbb R^n$ with
center at $\uu$, i.e.
\begin{equation}\label{cpeps}
c'_{\varepsilon}(a)=\varepsilon a+ (1-\varepsilon)\uu,\text{ for all }a\in\mathbb R^n.
\end{equation}
Similarly to $c_{\varepsilon}$, note that $c'_{\varepsilon}$ is convex and bijective and its inverse is $c'_{\frac{1}{\varepsilon}}$.

Also, for $\delta \in (0,1]$, we let $\FF_{\geq \delta}=\FF\cap [\delta,1]^n$.

\begin{lema}\label{lem2}
Let $\varepsilon\in (0,1]$ and consider $c'_{\varepsilon}|_{\FF}:\FF \to \mathbb R^n$, the restriction of $c'_{\varepsilon}$. Then:

\begin{enumerate}
\item[(1)] The image $c'_{\varepsilon}|_{\FF}$ is $\Imm(c'_{\varepsilon}|_{\FF})=\FF_{\geq 1-\varepsilon}$.
\item[(2)] $c'_{\frac{1}{\varepsilon}}|_{\FF_{\geq 1-\varepsilon}}: \FF_{\geq 1-\varepsilon} \to \FF$ is the inverse of 
           $c'_{\varepsilon}|_{\FF}:\FF \to \FF_{\geq 1-\varepsilon}$.
\end{enumerate}
\end{lema}

\begin{proof}
(1) Indeed, if $a\in \FF$, it is clear that $a':=c'_{\varepsilon}(a)\in\HH = [0,1]^n$. Moreover, if $a_i=1$, then $a'_i=1$.
    Hence, $a'\in\FF$ as well. The injectivity of $c'_{\varepsilon}|_{\FF}$ is clear.
    It is clear that $a'_i\geq 1-\varepsilon$ for all $i\in [n]$, where $a'=c_{\varepsilon}(a)$ and $a\in\FF$ as above.
    Conversely, if $a'\in \FF_{\geq 1-\varepsilon}$, if we let $a:=c_{\frac{1}{\varepsilon}}(a')$ then $c_{\varepsilon}(a')=a$.
		
(2) It is obvious.
\end{proof}

\begin{prop}\label{diag}
The following diagram is commutative:
$$\begin{tikzcd}
 \FF \arrow{r}{c'_{\frac{1}{n-1}}} \arrow[swap]{d}{\Psi_0} & \FF_{\geq \frac{n-2}{n-1}} \arrow{d}{\Psi_0} \\%
 \PP \arrow{r}{c_{\frac{1}{n-1}}}& \overline{\PP}
 \end{tikzcd}$$
Moreover, all the above maps are bijective. (Of course, we implicitly consider the appropriate restrictions)
\end{prop}

\begin{proof}
Let $a\in\FF$. From Proposition \ref{p4}(4), we have
$$c_{\frac{1}{n-1}}(\Psi_0(a))=\Psi(a)=\frac{1}{n-1}a - \frac{1}{n(n-1)}\left( \sum_{i\in[n]} a_i\right)\uu + \frac{1}{n}\uu.$$
On the other hand, we have
$$\Psi_0(c'_{\frac{1}{n-1}}(a))=\Psi_0\left(\frac{1}{n-1}a + \frac{n-2}{n-1}\uu\right),$$
hence, using \eqref{psi0}, we get $c_{\frac{1}{n-1}}(\Psi_0(a))=\Psi_0(c'_{\frac{1}{n-1}}(a))$, so the diagram is commutative.
The bijectivity of the maps follows from Lemma \ref{lem1}, Proposition \ref{p3} and Lemma \ref{lem2}.
\end{proof}

\begin{dfn}
Let $\delta\in[0,1)$. We define the category $\Covering_{\geq \delta}$ as the full subcategory of $\Covering$
whose objects $(X,(A_i)_{i\in I})$ satisfy the condition
$$ A_i(x)\geq \delta\text{ for all }i\in I,x\in X.$$
Let $\varepsilon\in (0,1]$. We define the functor $\mathbf{C}_{\varepsilon}: \Covering \to \Covering_{\geq 1-\varepsilon}$ as follows:
\begin{enumerate}
\item[(1)] On objects: $\mathbf{C}_{\varepsilon}((X,(A_i)_{i\in I})):=(X,(B_i)_{i\in I})$, where
$$ B_i(x)=\varepsilon A_i(x)+1-\varepsilon,\;\text{ for all }i\in I,x\in X.$$
\item[(2)] On morphisms: $\mathbf{C}_{\varepsilon}((f,\rho)):=(f,\rho)$.
\end{enumerate}
Also, we define the functor $\mathbf{D}_{\varepsilon}:\Covering_{\geq 1-\varepsilon} \to \Covering$ as follows:
\begin{enumerate}
\item[(1)] On objects: $\mathbf{D}_{\varepsilon}((X,(B_i)_{i\in I}):=((X,(A_i)_{i\in I})$, where
$$ A_i(x)=\frac{1}{\varepsilon} B_i(x)+1-\frac{1}{\varepsilon},\;\text{ for all }i\in I,x\in X.$$
\item[(2)] On morphisms: $\mathbf{D}_{\varepsilon}((f,\rho)):=(f,\rho)$.
\end{enumerate}
\end{dfn}

With the above notations, we have:

\begin{prop}\label{c_eps}
The functors $\mathbf{C}_{\varepsilon}$ and $\mathbf{D}_{\varepsilon}$ are well defined and fully faithful.
Moreover, they induce an isomorphism of categories between $\Covering$ and  $\Covering_{\geq 1-\varepsilon}$.
\end{prop}

\begin{proof}
Let $(X,(A_i)_{i\in I})\in\Ob(\Covering)$.
It is easy to see that $$\mathbf{C}_{\varepsilon}((X,(A_i)_{i\in I}))=  ((X,(B_i)_{i\in I})\in \Ob(\Covering_{\geq 1-\varepsilon}).$$ 
Now, assume that 
$(f,\rho):(X,(A_i)_{i\in I})\to (Y,(A'_j)_{j\in J})$ is a morphism in $\Covering$, i.e.
$$A_i(x)\leq A'_{\rho(i)}(f(x)),\text{ for all }i\in I, x\in X.$$
It follows that $B_i(x)=\varepsilon A_i(x) + 1-\varepsilon \leq \varepsilon A'_{\rho(i)}(f(x)) + 1 - \varepsilon = B'_{\rho(i)}(f(x)),$
where $(Y,(B'_j)_{j\in J})=\mathbf{C}_{\varepsilon}((Y,(A'_j)_{j\in J}))$. 
Therefore, $\mathbf{C}_{\varepsilon}((f,\rho))=(f,\rho)$ is a morphism in
$\Covering_{\geq 1-\varepsilon}$. Hence $\mathbf{C}_{\varepsilon}$ is a functor.
Also, it is clear that $\mathbf{C}_{\varepsilon}$ is fully faithful.

Similarly, one can easily check that $\mathbf{D}_{\varepsilon}$ is also a fully faithful functor,
$\mathbf{C}_{\varepsilon}\circ \mathbf{D}_{\varepsilon} = \id_{\Ob(\Covering_{\geq 1-\varepsilon})}$ and
$\mathbf{D}_{\varepsilon}\circ \mathbf{C}_{\varepsilon} = \id_{\Ob(\Covering)}$. Thus, the proof is complete.
\end{proof}

\begin{dfn}\label{covn}
We define the categories:
\begin{enumerate}
\item[(1)] $\Covering[n]$, as the full subcategory of $\Covering$ whose objects are coverings with $n$ fuzzy sets.
\item[(2)] $\Covering[n]_{\geq \frac{n-2}{n-1}}$, as the full subcategory of $\Covering[n]$ whose objects are coverings 
           from $\Covering_{\geq \frac{n-2}{n-1}}$.
\item[(3)] $\Partition[n]$, as the full subcategory of $\Partition$ whose objects are partitions with $n$ fuzzy sets.
\item[(4)] $\Partition[n]_{\overline{\PP}}$, as the full subcategory of $\Partition$ whose objects are $(X,(B_i)_{i\in[n]})$
           such that $(B_1(x),\ldots,B_n(x))\in \overline{\PP}$ for all $x\in X$.
\end{enumerate}
\end{dfn}

\begin{cor}
With the above notations, $\mathbf{C}_{\frac{1}{n-1}}:\Covering[n] \to \Covering[n]_{\geq \frac{n-2}{n-1}}$ is an isomorphism of categories.
\end{cor}

\begin{proof}
It follows from Proposition \ref{c_eps}.
\end{proof}

\begin{dfn}
We define the functor $\mathbf{F[n]}:\Covering[n] \to \Partition[n]_{\overline{\PP}}$, as follows:
\begin{enumerate}
\item[(1)] On objects: $\mathbf{F[n]}((X,(A_i)_{i\in[n]})):=(X,(B_i)_{i\in [n]})$, where 
      $$B_i(x)=\frac{1}{n-1}A_i(x) - \frac{1}{n(n-1)}\left( \sum_{i\in[n]} A_i(x)\right) + \frac{1}{n}\text{ for all }i\in[n],x\in X.$$
\item[(2)] On morphisms: $\mathbf{F[n]}((f,\rho)):=(f,\rho)$.
\end{enumerate}
We also define the functor $\mathbf{G[n]}:\Partition[n]_{\overline{\PP}} \to \Covering[n]$, as follows:
\begin{enumerate}
\item[(1)] On objects: $\mathbf{G[n]}((X,(B_i)_{i\in[n]})):=(X,(A_i)_{i\in [n]})$, where 
      $$A_i(x) = (n-1)(B_i(x)-\bigvee_{i\in[n]}B_i(x))+1\text{ for all }i\in[n],x\in X.$$
\item[(2)] On morphisms: $\mathbf{G[n]}((f,\rho)):=(f,\rho)$.
\end{enumerate}
\end{dfn}

\begin{teor}\label{t2}
With the above notations, the functors $\mathbf{F[n]}$ and $\mathbf{G[n]}$ are well defined and 
and fully faithful. Moreover, they induce an isomorphism of categories between $\Covering[n]$ and 
$\Partition[n]_{\overline{\PP}}$.
\end{teor}

\begin{proof}
The fact that $\mathbf{F[n]}$ is well defined on the objects of $\Covering[n]$ follows from Proposition \ref{p4}.
Now, let $(f,\rho):(X,(A_i)_{i\in[n]}) \to (Y,(A'_i)_{i\in[n]})$ be a morphism in $\Covering[n]$, that is 
$$A_i(x)\leq A_{\rho(i)}(f(x))\text{ for all }i\in[n],x\in X.$$
Let $(X,(B_i)_{i\in[n]}):=\mathbf{F[n]}((X,(A_i)_{i\in[n]}))$ and $(Y,(B'_i)_{i\in[n]}):=\mathbf{F[n]}((Y,(A'_i)_{i\in[n]}))$.
Then $$B_i(x)-\bigvee_{i\in [n]}B_i(x) = \frac{1}{n-1}(A_i(x) - 1) \leq \frac{1}{n-1}(A_{\rho(i)}(f(x)) - 1) = 
B_{\rho(i)}(f(x))-\bigvee_{i\in[n]} B_i(f(x)),$$
hence $(f,\rho):(X,(B_i)_{i\in[n]}) \to (Y,(B'_i)_{i\in[n]})$ is a morphism in $\Partition[n]_{\overline{\PP}}$.
Thus $\mathbf{F[n]}$ is a functor. Also, it is clear that $\mathbf{F[n]}$ is fully faithful.

Similarly, one can easily check that $\mathbf{G[n]}$ is also a fully faithful functor,
$\mathbf{G[n]}\circ \mathbf{F[n]} = \id_{\Ob(\Covering[n])}$ and
$\mathbf{F[n]}\circ \mathbf{G[n]} = \id_{\Ob(\Partition[n]_{\overline{\PP}})}$. Thus, the proof is complete.
\end{proof}

\section{A third bijection between partitions and finite coverings}

We use the notations from the previous section.
Let $b\in\PP\setminus\{g\}$ and consider the line $\ell$
determined by $g$ and $b$. 
We define 
$$\alpha(b):=\max\{\alpha\geq 0\;:\;\alpha\cdot (b-g)+g\in\PP\}.$$
Note that $\alpha(b)\cdot (b-g)+g \in\partial \PP$, where $\partial \PP$ is the border of $\PP$. Indeed, this follows
from the fact that $\alpha(b)\cdot (b-g)+g \in \ell\cap \PP$, $\PP$ is convex and the definition of $\alpha(b)$.

\begin{lema}\label{alfa}
With the above notations, we have $$\alpha(b)=\frac{1}{\bigvee_{i\in[n]}b_i - \bigwedge_{i\in[n]}b_i}.$$
\end{lema}

\begin{proof}
First, note that $\alpha(b)\geq 1$. Let $\alpha\geq 1$ such that $\alpha\cdot (b-g)+g\in\PP$.
We consider the inverse of the map $\Psi_0$, defined in \eqref{psi0}, namely $\Psi_0^{-1}:\PP\to\HH$.
As in Proposition \ref{p3}, we have
$$\Psi_0^{-1}(b)=b+(1-\bigvee_{i\in[n]}b_i)\uu.$$
It follows that
$$\Psi_0^{-1}(\alpha(b-g)+g)=\alpha\cdot (b-g)+g + (1-\alpha\bigvee_{i\in[n]}b_i - \frac{1}{n}(1-\alpha))\uu \in \FF.$$
Therefore, for all $i\in [n]$, we have
\begin{equation}\label{alpha}
 \alpha\cdot b_i - \alpha\bigvee_{i\in[n]}b_i + 1 \geq 0.
\end{equation}
From \eqref{alpha} it follows that $\alpha\leq \frac{1}{\bigvee_{i\in[n]} b_i - b_i}$ for all $i\in[n]$ with $b_i<\bigvee_{i\in[i]}b_i$.
The largest value of $\alpha$ which satisfies these conditions is
$$\alpha(b)=\frac{1}{\bigvee_{i\in[n]}b_i - \bigwedge_{i\in[n]}b_i},$$
as required.
\end{proof}

We let $\beta(b):=\max\{\beta\geq 0\;:\;\beta\cdot (b-g)+g\in\KK\}$. Note also that $\beta(b)(b-g)+g\in\partial\KK$. 

\begin{lema}\label{beta}
With the above notations, we have $$\beta(b)=\frac{1}{1-n\bigwedge_{i\in[n]}b_i}.$$
\end{lema}

\begin{proof}
Note that $\beta(b)\geq \frac{1}{n-1}$. Indeed, if $b\in\PP$ then $b_i \in [-\frac{n-2}{n},1]$ for all $i\in [n]$.
By straightforward computations, it follows that the coordinates of $d:=\frac{1}{n-1}(b-g)+g$ are nonzero. 
Since $d\in \PP\subset\Pi$, it follows that $d\in\KK$.

Let $\beta\geq \frac{1}{n-1}$ such that $\beta\cdot (b-g)+g\in\KK$. It follows that
$$\beta\cdot \left(b_i-\frac{1}{n}\right)+\frac{1}{n}\geq 0\text{ for all }i\in[n].$$
The largest value of $\beta$ which satisfies these conditions is
$$\beta(b) = \frac{\frac{1}{n}}{\frac{1}{n}-\bigwedge_{i\in[n]}b_i} = \frac{1}{1-n\bigwedge_{i\in[n]}b_i},$$
as required.
\end{proof}

\begin{prop}\label{p5}
With the above notations, the map 
$$\psi:\PP \to \KK,\; \psi(b)= \begin{cases} \frac{\beta(b)}{\alpha(b)}(b-g)+g,& b\neq g \\ g,& b=g  \end{cases} = 
\begin{cases} \frac{\bigvee_{i\in[n]}b_i - \bigwedge_{i\in[n]}b_i}{1-n\bigwedge_{i\in[n]}b_i}(b-g)+g,& b\neq g \\ g,& b=g \end{cases},$$ 
is bijective. Its inverse is $\psi^{-1}:\KK\to\PP$, defined by
$$ \psi^{-1}(b')= \begin{cases} \frac{1-n\bigwedge_{i\in[n]}b'_i}{\bigvee_{i\in[n]}b'_i - \bigwedge_{i\in[n]}b'_i}(b'-g)+g,& b'\neq g \\ g,& b'=g \end{cases}.$$
\end{prop}

\begin{proof}
First, note that $\psi$ is well defined. Indeed, if $b\in\PP\setminus\{g\}$ then $\beta(b)(b-g)+g\in \KK$. 
Since $\alpha(b)\geq 1$ and $\KK$ is
convex, it follows that $\psi(b) \in \KK$, as required.
The last expression of $\psi(b)$ follows from Lemma \ref{alfa} and Lemma \ref{beta}.

Let $b\in \PP\setminus\{g\}$ such that $\psi(b)=\psi(g)=g$. Note that, since $b\neq g$, it follows that $\bigvee_{i\in[n]}b_i>\frac{1}{n}$ and $\bigwedge_{i\in[n]}b_i<\frac{1}{n}$.
Without any loss of generality, we can assume that $b_1\geq b_2\geq \cdots \geq b_n$, hence $b_1=\bigvee_{i\in[n]}b_i > \bigwedge_{i\in[n]}b_i=b_n$. Moreover, we have that
\begin{equation}\label{biep}
 b':=\psi(b) = \frac{b_1-b_n}{1-nb_n}(b-g)+g.
\end{equation}
Since $b_1>b_n$, from \eqref{biep} it follows that $b'_1>b'_n$ and thus $b'=\psi(b) \neq g=\psi(g)$.

Now, let $c\in \PP\setminus\{g\}$ such that $\psi(b)=\psi(c)$. It follows that
\begin{equation}\label{bc}
\frac{\bigvee_{i\in[n]}b_i - \bigwedge_{i\in[n]}b_i}{1-n\bigwedge_{i\in[n]}b_i}(b-g) = \frac{\bigvee_{i\in[n]}c_i - \bigwedge_{i\in[n]}c_i}{1-n\bigwedge_{i\in[n]}c_i}(c-g).
\end{equation}
Since $b_1\geq b_2\geq \ldots b_n$, from \eqref{bc} it follows that $c_1\geq c_2\geq \cdots \geq c_n$, hence $c_1=\bigvee_{i\in[n]}c_i > \bigwedge_{i\in[n]}c_i=c_n$.
From \eqref{bc} we get
\begin{equation}\label{bcc}
 \frac{b_1 - b_n}{1-nb_n}(1-nb_i) = \frac{c_1 - c_n}{1-nc_n}(1-nc_i)\text{ for all }i\in[n].
\end{equation}
Taking $i=n$ in \eqref{bcc}, we get $b_1-b_n=c_1-c_n$ and therefore, from \eqref{bcc} we get
\begin{equation}\label{bccc}
 \frac{1 - nb_i}{1-nb_n} = \frac{1 - nc_i}{1-nc_n} \text{ for all }i\in[n-1].
\end{equation}
If $b_n=c_n$, then, from \eqref{bccc} it follows that $b=c$. Assume $b_n\neq c_n$. Since $c_1-b_1=c_n-b_n$, from \eqref{bccc} we get
$$ \frac{1 - nb_1}{1-nb_n} = \frac{1 - nc_1}{1-nc_n} = \frac{c_1-b_1}{c_n-b_n}=1, $$
thus $b_1=b_n$ which implies $b_1=\cdots=b_n=\frac{1}{n}$, hence $b=g$, a contradiction!
From all the above, it follows that the map $\psi$ is injective. 

In order to check that $\psi$ is surjective, it is enough to show
that for $b'\in \KK\setminus\{g\}$, if we set 
$$b:=\frac{1-n\bigwedge_{i\in[n]}b'_i}{\bigvee_{i\in[n]}b'_i - \bigwedge_{i\in[n]}b'_i}(b'-g)+g,$$
it follows that $b\in\PP\setminus\{g\}$ and $\psi(b)=b'$. We leave this as an exercise!
\end{proof}

\begin{prop}\label{c222}
The composite map 
$\Psi_1:\FF \to \KK$, $\Psi_1:=\psi\circ\Psi_0|_{\FF}$, is bijective. 
Moreover, for $a\in\FF$, we have\
$$\Psi_1(a)= \begin{cases} \frac{1-\bigwedge_{i\in[n]}a_i}{\sum_{i\in[n]}a_i - n\bigwedge_{i\in[n]}a_i}
\left(a - (\sum_{i \in [n]} a_i)g \right) + g,& a\neq g \\ g,& a=g \end{cases}.$$
Moreover, the inverse of $\Psi_1$ is $\Psi_1^{-1}:\KK \to \FF$, $\Psi_1^{-1}=\Psi_0|_{\FF}^{-1}\circ\psi^{-1}$, where
$$\Psi_1^{-1}(b)= \begin{cases} \frac{1-n\bigwedge_{i\in[n]}b_i}{\bigvee_{i\in[n]}b_i - \bigwedge_{i\in[n]}b_i }(b-g)
+ \left( 1- \frac{1-n\bigwedge_{i\in[n]}b_i}{\bigvee_{i\in[n]}b_i}(\bigvee_{i\in[n]}b_i -\frac{1}{n} ) \right)u
,& b\neq g \\ g,& b=g \end{cases}.$$
\end{prop}

\begin{proof}
From Proposition \ref{p3} and Proposition \ref{p5}, it follows that $\Psi_1$ is bijective.
Let $a\in\FF$. Then by \eqref{psi0}, we have $\Psi_1(a)=\psi(a+(1-\sum_{i\in [n]} a_i)g)$. The required formula follows 
Proposition \ref{p5} by straightforward computations. Similarly, the formula for $\Psi^{-1}(b)$ follows
 from Proposition \ref{p3} and Proposition \ref{p5}.
\end{proof}

\begin{dfn}
If $(X,(A_i)_{i\in[n]})$ is a covering, then we let $\mathbf{F}_3((X,(A_i)_{i\in[n]})):=(X,(B_i)_{i\in[n]})$, where:
$$(B_1(x),\ldots,B_n(x)):=\Psi_1(A_1(x),\ldots,A_n(x)),\text{ for all }x\in X.$$
If $(X,(B_i)_{i\in[n]})$ is a partition, then we let $\mathbf{G}_3((X,(B_i)_{i\in[n]})):=(X,(A_i)_{i\in[n]})$, where:
$$(A_1(x),\ldots,A_n(x)):=\Psi_1^{-1}(B_1(x),\ldots,B_n(x)),\text{ for all }x\in X.$$
Also, we define $\mathbf{F}_3((X,X))=(X,X)$ and $\mathbf{G}_3((X,X))=(X,X)$.
\end{dfn}

\begin{teor}\label{t222}
With the above notations, the maps $\mathbf{F}_3:\Ob(\fCovering)\to \Ob(\Partition)$ and 
$\mathbf{G}_3:\Ob(\Partition)\to \Ob(\fCovering)$ are bijective and inverse to each other.
\end{teor}

\begin{proof}
 The proof is similar to the proof of Theorem \ref{c22}, using Proposition \ref{c222}.
\end{proof}

\section{Case $n=2$}

In this section we detail the results obtained in the previous sections in the case $n=3$.

We have $\HH:=[0,1]^2$, $\FF=(\{1\}\times [0,1]) \cup ([0,1]\times\{1\})$ and
$\KK=\{x\in \HH\;:\;x_1+x_2=1\}$. Note that $\KK$ is the diagonal of the square $\HH$.

The map $\Phi:\KK\to\FF$ from Proposition \ref{p1} is given by
$$\Phi(b_1,b_2)=(b_1+1-\max\{b_1,b_2\},b_2+1-\max\{b_1,b_2\}).$$
Note that $\Phi$ is a bijection and its inverse, $\Phi^{-1}:\FF\to\KK$, is given by
$$\Phi^{-1}(a_1,a_2)=\frac{1}{2}(a_1-a_2+1,a_2-a_1+1).$$
\begin{center}
\includegraphics[width=0.3\textwidth]{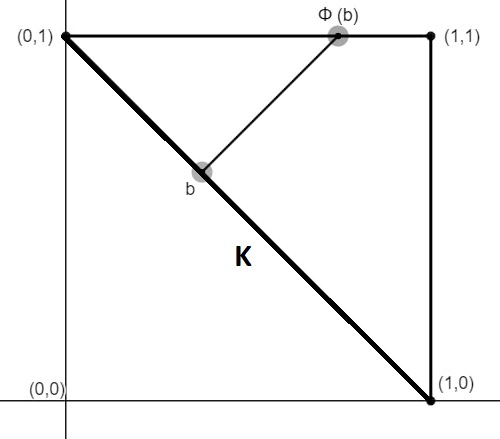}

\small{Figure 1: Representation of the bijection $\Phi$}
\end{center}
According to Theorem \ref{t1}, there is
an isomorphism of categories $\mathbf F:\Partition[2]\to\Covering[2]$, defined
\begin{itemize}
\item On objects: If $(X,(B_1,B_2))$ is a partition of $X$, then $\mathbf F((X,(B_1,B_2))):=(X,(A_1,A_2))$, 
      where $A_i(x)=B_i(x)+1-\max\{B_1(x),B_2(x)\}$, for $i=1,2$.
\item On morphisms: $\mathbf F((f,\rho))=(f,\rho)$.
\end{itemize}
Its inverse is $\mathbf G:\Covering[2]\to\Partition[2]$, defined
\begin{itemize}
\item On objects: If $(X,(A_1,A_2))$ is a covering of $X$, then $\mathbf G((X,(A_1,A_2))):=(X,(B_1,B_2))$, 
      where $B_1(x)=\frac{1}{2}(A_1(x)-A_2(x)+1)$ and $B_2(x)=\frac{1}{2}(A_2(x)-A_1(x)+1)$.
\item On morphisms: $\mathbf G((f,\rho))=(f,\rho)$.
\end{itemize}

Note that $\Phi=\Phi_1=\Phi_2$, where $\Phi_1$ was defined in Proposition \ref{p22} and $\Phi_2$ in Proposition \ref{p2}. 
Also, $\Psi=\Phi^{-1}$, where $\Psi$ is the map defined in Proposition \ref{p4}. Hence, Theorem \ref{t2} gives nothing new
in this case. Finally, the map $\psi$ defined in Proposition \ref{p5} is the identity map.

\section{Case $n=3$}

In this section we detail the results obtained in the previous sections in the case $n=3$.

We have $\HH:=[0,1]^3$, $\FF=(\{1\}\times [0,1]^2) \cup ([0,1]\times\{1\}\times[0,1])\cup ([0,1]^2\times\{1\})$ and
$\KK=\{x\in \HH\;:\;x_1+x_2+x_3=1\}$. Note that $\KK$ is an (equilateral) triangle with the vertices 
$(1,0,0)$, $(0,1,0)$ and $(0,0,1)$, i.e. $\KK$ is the convex hull of $(1,0,0)$, $(0,1,0)$ and $(0,0,1)$.
\begin{center}
\includegraphics[width=0.4\textwidth]{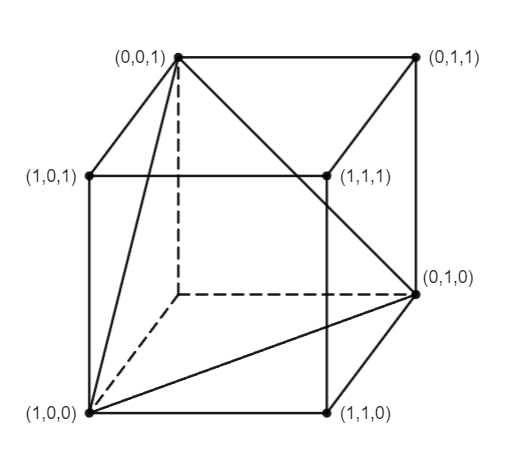}

\small{Figure 2: Representation of $\HH$ and $\KK$}
\end{center}
The map $\Phi:\KK\to\FF$ from Proposition \ref{p1} is given by
$$\Phi(b_1,b_2,b_3)=(b_1+1-\max\{b_1,b_2,b_3\},b_2+1-\max\{b_1,b_2,b_3\},b_3+1-\max\{b_1,b_2,b_3\}).$$
The image of $\Phi$ is 
\begin{equation}\label{greeny}
\GG=\{(a_1,a_2,a_3)\in\FF\;:2a_1-a_2-a_3+1\geq 0,\;2a_2-a_1-a_3+1\geq 0,\;2a_3-a_2-a_3+1\geq 0 \;\}.
\end{equation}
If we denote $\GG_i=\GG\cap \{x_i=1\}$ for $i\in [3]$, then it is easy to see that 
$$\GG_1=\{(1,a_2,a_3)\;:\;a_2 \leq 2a_3,\;a_3\leq 2a_2 \}.$$
\begin{center}
\includegraphics[width=0.35\textwidth]{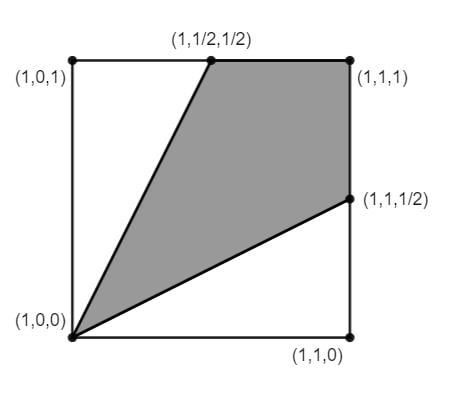}

\small{Figure 3: Representation of $\GG_1$}
\end{center}
Similar formulas can be provided for $\GG_2$ and $\GG_3$. Also, it is clear that $\GG=\GG_1\cup \GG_2\cup \GG_3$.

Similarly to \eqref{greeny}, a "good" covering of a set $X$ with three fuzzy sets is a pair $(X,(A_i)_{i\in [3]})$ such that,
for all $x\in X$, we have \small
$$ 2A_1(x)-A_2(x)-A_3(x)+1\geq 0,\;2A_2(x)-A_1(x)-A_3(x)+1\geq 0,\;2A_3(x)-A_1(x)-A_2(x)+1\geq 0.$$ \normalsize
The isomorphism of categories $\mathbf F:\Partition\to\gCovering$ from Theorem \ref{t1} is given
on objects by $\mathbf F((X,(B_i)_{i\in [3]}))=(X,(A_i)_{i\in [3]})$, where
$$A_i(x)=B_i(x)+1-\max\{B_1(x),B_2(x),B_3(x)\}\text{ for all }i\in[3],x\in X.$$
Also, $\mathbf G:\gCovering \to \Partition$ is given on objects by 
$\mathbf G((X,(A_i)_{i\in [3]}))=(X,(B_i)_{i\in [3]})$, where
\begin{equation}\label{phiminus}
B_i(x)=A_i(x)-\frac{1}{3}(A_1(x)+A_2(x)+A_3(x))+\frac{1}{3} \text{ for all }i\in[3],x\in X.
\end{equation}
The bijection given in Proposition \ref{p22} is $\Phi_1:\FF\to\GG$, given by
$$ \Phi_1(c_1,c_2,c_3)=\left(\frac{2c_1+c_2+c_3-1}{c_1+c_2+c_3},\frac{c_1+2c_2+c_3-1}{c_1+c_2+c_3},
\frac{c_1+c_2+2c_3-1}{c_1+c_2+c_3} \right).$$
Its inverse is $\Phi^{-1}:\GG\to\FF$, where
$$\Phi^{-1}(a_1,a_2,a_3)=\left(\frac{2a_1-a_2-a_3+1}{4-a_1-a_2-a_3},\frac{-a_1+2a_2-a_3+1}{4-a_1-a_2-a_3},
\frac{-a_1-a_2+2a_3+1}{4-a_1-a_2-a_3}, \right).$$
According to Corollary \ref{c2}, $\Phi_1$, induced a bijection between the coverings with 3 sets and the good coverings with 3 sets,
More precisely, if $(X,(C_i)_{i\in[3]})$ is a covering of $X$, then $(X,(A_i)_{i\in [3]})$ is a good covering of $X$,
where $(A_1(x),A_2(x),A_3(x)):=\Phi_1(C_1(x),C_2(x),C_3(x))$ for all $x\in X$. Obviously, we have
$(C_1(x),C_2(x),C_3(x))=\Phi_1^{-1}(A_1(x),A_2(x),A_3(x))$ for all $x\in X$.

Now, let $e\in S_3$ be the identical permutation. Then 
$$\FF_e:=\{(c_1,c_2,c_3)\in\FF\;:\;1=c_1\geq c_2\geq c_3\}\text{ and }\GG_e:=\{(a_1,a_2,a_3)\in\FF_e\;:\;2a_3\geq a_2\}.$$
The map $\Phi_{2,e}:\GG_e\to\FF_e$ given in Lemma \ref{fis} has the expression
$$\Phi_{2,e}(a_1,a_2,a_3)=(a_1,a_2,2a_3-a_2),\text{ for all }(a_1,a_2,a_3)\in\GG_e.$$
Its inverse is $\Phi_{2,e}^{-1}:\FF_e\to \GG_3$ and has the expression
$$\Phi_{2,e}^{-1}(c_1,c_2,c_3)=\left( c_1,c_2,\frac{1}{2}c_2+\frac{1}{2}c_3 \right),\text{ for all }(c_1,c_2,c_3)\in\FF_e.$$
For $\sigma\in S_3$, $\FF_{\sigma}$, $\GG_{\sigma}$ and $\Phi_{2,\sigma}$ are similarly defined and constructed. For instance,
if $\sigma=(231)$, then 
\begin{align*}
& \FF_{\sigma}=\{(c_1,c_2,c_3)\in\FF\;:\;c_2\geq c_3\geq c_1\},\;\GG_{\sigma}=\{(a_1,a_2,a_3)\in\FF_{\sigma}\;:\;2a_1\geq a_3\} \\
& \Phi_{2,\sigma}(a_1,a_2,a_3) = (2a_1-a_3,a_2,a_3),\;\Phi_{2,\sigma}^{-1}(c_1,c_2,c_3)=\left(\frac{1}{2}c_3+\frac{1}{2}c_1,c_2,c_3\right).
\end{align*}
Setting $\Phi_2(a):=\Phi_{2,\sigma}(a)$ for $a\in \GG_{\sigma}$, according to Proposition \ref{p2} and Proposition \ref{c22}, we obtain
a bijection $\Phi_2:\GG\to\FF$ which induced a bijection between the good coverings with 3 sets and the coverings with 3 sets.

More precisely, if $(X,(A_i)_{i\in[3]})$ is a good covering of $X$, then $(X,(C_i)_{i\in [3]})$ is a covering of $X$,
where $(C_1(x),C_2(x),C_3(x)):=\Phi_2(A_1(x),A_2(x),A_3(x))$ for all $x\in X$.
Obviously, we have $(A_1(x),A_2(x),A_3(x))=\Phi_2^{-1}(C_1(x),C_2(x),C_3(x))$ for all $x\in X$.

The map $\Psi_0:\HH\to\Pi$, given in \eqref{psi0}, is defined by
$$\Psi_0(a_1,a_2,a_3)=\left( \frac{2a_1-a_2-a_3+1}{3},\frac{-a_1+2a_2-a_3+1}{3},\frac{-a_1-a_2+2a_3+1}{3}, \right).$$
Let $\PP:=\Imm(\Psi_0)$ and note that $\PP=\Psi_0(\HH)$. Then $\PP$ is a hexagon contained in the plane $\Pi:x_1+x_2+x_3=1$,
with the vertices 
$$ (1,0,0), \left(\frac{2}{3},\frac{2}{3},-\frac{1}{3} \right),   (0,1,0), \left(\frac{2}{3},-\frac{1}{3},\frac{2}{3} \right),   (0,0,1),
 \left(-\frac{1}{3},\frac{2}{3},\frac{2}{3} \right).$$
According to Proposition \ref{p3}, $\Psi_0:\HH\to\PP$ is bijective and
$$\Psi_0^{-1}(b_1,b_2,b_3)=(b_1+1-\max\{b_1,b_2,b_3\} ,b_2+1-\max\{b_1,b_2,b_3\},b_3+1-\max\{b_1,b_2,b_3\})$$
In Proposition \ref{p4}, the map $\Psi:\HH\to\KK$ was introduced and it has the expression
$$\Psi(a_1,a_2,a_3)=\left( \frac{2a_1-a_2-a_3+2}{6}, \frac{-a_1+2a_2-a_3+2}{6}, \frac{-a_1-a_2+2a_3+2}{6}, \right).$$
The image of $\Psi$ is $\overline{\PP}$, which is a hexagon with the vertices
$$ \left( \frac{2}{3}, \frac{1}{6}, \frac{1}{6} \right), \left( \frac{1}{2}, \frac{1}{2}, 0 \right),
 \left( \frac{1}{6}, \frac{2}{3}, \frac{1}{6} \right), \left(  \frac{1}{2}, 0, \frac{1}{2} \right), 
\left( \frac{1}{6}, \frac{1}{6}, \frac{2}{3} \right), \left( 0, \frac{1}{2}, \frac{1}{2} \right). $$
\begin{center}
\includegraphics[width=0.5\textwidth]{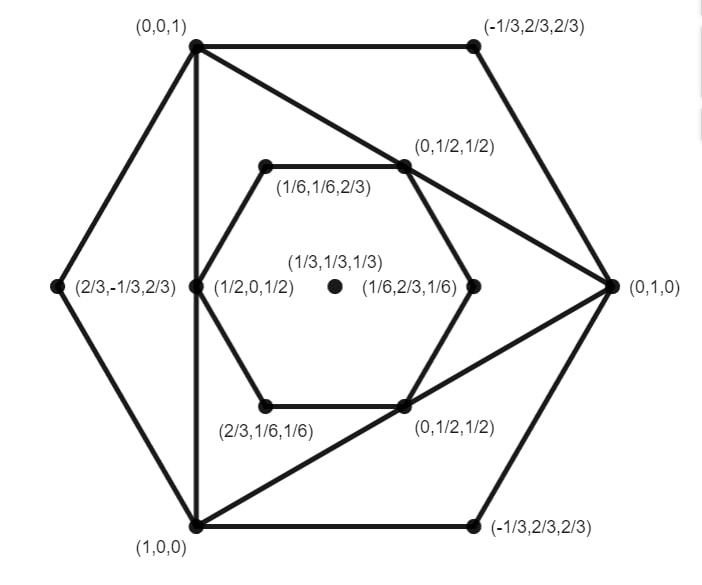}

\small{Figure 4: Representation of $\PP$, $\KK$ and $\overline{\PP}$}
\end{center}
Note that $\overline{\PP}$ and $\PP$ have the same center at $(\frac{1}{3},\frac{1}{3},\frac{1}{3})$.
$\KK$ is contained in $\PP$ and tangent to the border of $\PP$ in $(1,0,0)$, $(0,1,0)$ and $(0,0,1)$.
Also, $\overline{\PP}$ is contained in $\KK$ and tangent to the border of $\KK$ in $\left( \frac{1}{2}, \frac{1}{2}, 0 \right)$,
$\left( \frac{1}{2}, 0, \frac{1}{2} \right)$ and $\left( 0,\frac{1}{2}, \frac{1}{2} \right)$.

The inverse of $\Psi$ is $\Psi^{-1}:\PP\to \HH$, has the expression:
$$\Psi^{-1}(b_1,b_2,b_3)=( 2b_1-2\max\{b_1,b_2,b_3\}+1, 2b_2-2\max\{b_1,b_2,b_3\}+1, 2b_3-2\max\{b_1,b_2,b_3\}+1  ) .$$
By Definition \ref{covn}, we consider $\Covering[3]$, the category of coverings with $3$ sets, $\Covering[3]_{\geq\frac{1}{2}}$, 
the full subcategory of $\Covering[3]$ whose objects $(X,(A_i)_{i\in[3]})$ satisfy the condition $A_i(x)\geq\frac{1}{2}$ for 
all $i\in[3]$ and $x\in X$, $\Partition[3]$, the category of partitions with $3$ sets and $\Partition[3]_{\overline{\PP}}$

According to Proposition \ref{c_eps}, we have the isomorphism of categories 
$C_{\frac{1}{2}}:\Covering[3]\to\Covering[3]_{\geq \frac{1}{2}}$, defined on objects by 
      $C_{\frac{1}{2}}((X,(A_i)_{i\in[3]}) = (X, (B_i)_{i\in[3]})$,
			$$(B_1(x),B_2(x),B_3(x)):=\left(\frac{1}{2}A_1(x)+\frac{1}{2},\frac{1}{2}A_2(x)+\frac{1}{2},\frac{1}{2}A_3(x)+\frac{1}{2}\right)
			  \text{ for all }x\in X.$$
According to Theorem \ref{t2}, we have the isomorphism of categories
     $\mathbf{F[3]}:\Covering[3]\to\Partition[3]_{\overline{\PP}}$, defined on objects by 
      $\mathbf{F[3]}((X,(A_i)_{i\in[3]})=(X, (B_i)_{i\in[3]})$, where 
			\begin{align*}
			& (B_1(x),B_2(x),B_3(x)):=\left( \frac{2A_1(x)-A_2(x)-A_3(x)+2}{6}, \frac{-A_1(x)+2A_2(x)-A_3(x)+2}{6}, \right.\\
			& \left. \frac{-A_1(x)-A_2(x)+2A_3(x)+2}{6} \right), \text{ for all }x\in X.
			\end{align*}
			\normalsize

According to Proposition \ref{c222}, we have the bijection $\Psi_1:\FF\to\KK$, defined by
$$\Psi_1(a)= \begin{cases} \frac{1- \min\{a_1,a_2,a_3\}}{a_1+a_2+a_3 - 3 \min\{a_1,a_2,a_3\}}
\left( \frac{2a_1-a_2-a_3}{3}, \frac{-a_1+2a_2-a_3}{3}, \frac{-a_1-a_2+2a_3}{3}\right) + g,& a\neq g \\ g,& a=g \end{cases}.$$
Also, the inverse of $\Psi_1$ is $\Psi_1^{-1}:\KK \to \FF$, where
$$\Psi_1^{-1}(b)= \begin{cases} \frac{1-3\min\{b_1,b_2,b_3\}}{\max\{b_1,b_2,b_3\} - \min\{b_1,b_2,b_3\} }(b-g)
+ \left( 1- \frac{1-3\min\{b_1,b_2,b_3\}}{\max\{b_1,b_2,b_3\}}(\max\{b_1,b_2,b_3\} -\frac{1}{3} ) \right)u
,& b\neq g \\ g,& b=g \end{cases}.$$
where $g=\frac{1}{3}\uu=(\frac{1}{3},\frac{1}{3},\frac{1}{3})$. According to Theorem \ref{t222}, $\Psi_1$ induces
a bijection between the coverings and the partitions of $X$ with $3$ set:

If $(X,(A_1,A_2,A_3))$ is a covering of $X$, then 
$\Psi_1((X,(A_1,A_2,A_3))):=(X,(B_1,B_2,B_3))$, where $(B_1(x),B_2(x),B_3(x))=\Psi_1(A_1(x),A_2(x),A_3(x))$, for all $x\in X$,
is a partition of $X$. 

Similarly, if $(X,(B_1,B_2,B_3))$ is a partition of $X$, then
$\Psi_1^{-1}((X,(B_1,B_2,B_3))):=(X,(A_1,A_2,A_3))$, where $(A_1(x),A_2(x),A_3(x))=\Psi_1^{-1}(B_1(x),B_2(x),B_3(x))$, for all $x\in X$,
is a partition of $X$.

\section{Conclusions and future work}

The use of category theory represents a recent approach in the study of coverings with fuzzy sets. 
Let $\Covering$ be the category of the category of fuzzy coverings; see \cite{N}. In this article, we introduced
$\Partition$, the category of fuzzy partitions and we studied some of its properties. Also, we established
an isomorphism of categories between $\Partition$ and a full subcategory of $\Covering$,
consisting in coverings with a finite number of fuzzy set which satisfy certain condition;
see Theorem \ref{t1}, and an isomorphism between $\Covering[n]$, the category of coverings with $n$ fuzzy sets, 
and a subcategory of $\Partition$, whose objects are partitions with $n$ sets which satisfies certain conditions;
see Theorem \ref{t2}. We used these isomorphism to deduce several bijections between partitions
and coverings with finitely many sets; see Theorem \ref{kungfu1}, Theorem \ref{c22} and Theorem \ref{t222}.

Future research directions include the study of fuzzy partitions with infinite sets and possible generalizations
of our results if we replace the interval $[0,1]$ with other lattices, for instance with quantales, as it was 
suggested by our referee. However, such generalizations would dilute the geometric flavour of our results, but
they might be useful in other unexpected ways.

\section*{Acknowledgement}

We would like to express out gratitude to the referee for his valuable remarks which helped us to improve the paper.

The first author was partially supported  by a grant of the Ministry of Research, Innovation and Digitization, 
CNCS - UEFISCDI, project number PN-III-P1-1.1-TE-2021-1633, within PNCDI III.

\subsection*{Data availability}

Data sharing not applicable to this article as no data sets were generated or analyzed during the current study.

\subsection*{Conflict of interest}

The authors have no relevant financial or non-financial interests to disclose.


\begin{thebibliography}{99}

\bibitem{baet} B.\ Baets, R.\ Mesiar, \textit{$\mathcal T$-partitions}, Fuzzy Sets and Systems \textbf{97} (1998), 211--223.

\bibitem{deer} L.\ D'eer, C.\ Cornelis, L.\ Godo, \textit{Fuzzy neighborhood operators based on fuzzy coverings}, Fuzzy Sets and Systems \textbf{312} (2017), 
               17--35.
							
\bibitem{deer2} L.\ D'eer, C.\ Cornelis, \textit{A comprehensive study of fuzzy covering-based rough set models: definitions, properties and interrelationships}, 
Fuzzy Sets and Systems \textbf{336} (2018), 1--26.

\bibitem{doob} J.\ L.\ Doob, \textit{Stochastic Processes}, Wiley, 1962.

\bibitem{aaa} A.\ A.\ A.\ Fora, M.\ M.\ M.\ Jaradat, M.\ A.\ H.\ Shakhatreh, W.\ A.\ A.\ Shatanawi, \textit{On fuzzy partitions}, International Journal of Pure and Applied Mathematics
              \textbf{vol 30, no. 4} (2006), 467--474.

\bibitem{MacL} S.\ MacLane, \textit{Categories for the working mathematician}, Graduate Texts in Mathematics, Vol. 5,  Springer-Verlag, New York, 1971. 

\bibitem{hug} D.\ Hug, W.\ Weil, \textit{Lectures on convex geometry}, Springer Nature Switzerland AG, 2020.

\bibitem{N} A.\ Neac\c su, \textit{On the category of fuzzy covering and related topics}, U.P.B. Sci. Bull., Series A, \textbf{Vol. 83, Iss. 2} (2021), 203--214.

\bibitem{N2} A.\ Neac\c su, \textit{On the category of fuzzy tolerance relations and related topics}, U.P.B. Sci. Bull., Series A, \textbf{Vol. 84, Iss. 4} (2022), 109--122.

\bibitem{ne} C.\ V. Negoi\c t\u{a}, D.\ A.\ R\u{a}lescu, \textit{Applications of Fuzzy Sets to Systems Analysis}, Springer Basel AG, 1975.

\bibitem{ng} H.\ T.\ Nguyen , N.\ R.\ Prasad, C.\ L.\ Walker, E.\ A.\ Walker, \textit{A First Course in Fuzzy and Neural Control}, Chapman \& Hall/CRC, 2003.

\bibitem{qz} Qing-Zhao Kong, Zeng-Xin Wei, \textit{Covering-based fuzzy rough sets}, Journal of Intelligent \& Fuzzy Systems \textbf{29} (2015), 2405--2411.

\bibitem{rusp} E.\ H.\ Ruspini, \textit{A new approach to clustering}, Information and Control \textbf{15} (1969), 22--32.

\bibitem{zadeh} L.\ A.\ Zadeh, \textit{Fuzzy sets}, Information and Control \textbf{8} (1965), 338--353.
\end{thebibliography}
\end{document}